\pdfoutput=1 

\documentclass[12pt,reqno]{amsart}

\usepackage{amssymb,amsmath,amsthm,graphicx,xcolor,mathtools,mathrsfs,tabularx,bbm,tikz,url}
\usepackage{mathabx}\changenotsign        
\usepackage{dsfont}
\usepackage[shortlabels]{enumitem} 
\usepackage{lmodern}

\usepackage[babel]{microtype}
\usepackage[british]{babel}
\usepackage[utf8]{inputenc}
\allowdisplaybreaks

\usepackage{hyperref} 
\hypersetup{
	colorlinks,
	linkcolor={red!60!black},
	citecolor={green!60!black},
	urlcolor={blue!60!black}
}

\usepackage[open,openlevel=2,atend]{bookmark}
\usepackage[abbrev]{amsrefs}

\def\ssign{\textsection\nobreak\hspace{1pt plus 0.3pt}}
\makeatletter
\let\origsection=\section 
\def\mysection{\@mystartsection{section}{1}\z@{.7\linespacing\@plus\linespacing}{.5\linespacing}{\normalfont\scshape\centering\ssign}}
\def\section{\@ifstar{\origsection*}{\mysection}}
\def\appendix{\par\c@section\z@ \c@subsection\z@
	\let\sectionname\appendixname
	\let\section=\origsection
	\def\thesection{\@Alph\c@section}}
\def\@mystartsection#1#2#3#4#5#6{\if@noskipsec \leavevmode \fi
	\par \@tempskipa #4\relax
	\@afterindenttrue
	\ifdim \@tempskipa <\z@ \@tempskipa -\@tempskipa \@afterindentfalse\fi
	\if@nobreak \everypar{}\else
	\addpenalty\@secpenalty\addvspace\@tempskipa\fi
	\@dblarg{\@mysect{#1}{#2}{#3}{#4}{#5}{#6}}}
\def\@mysect#1#2#3#4#5#6[#7]#8{\edef\@toclevel{\ifnum#2=\@m 0\else\number#2\fi}\ifnum #2>\c@secnumdepth \let\@secnumber\@empty
	\else \@xp\let\@xp\@secnumber\csname the#1\endcsname\fi
	\@tempskipa #5\relax
	\ifnum #2>\c@secnumdepth
	\let\@svsec\@empty
	\else
	\refstepcounter{#1}\edef\@secnumpunct{\ifdim\@tempskipa>\z@ \@ifnotempty{#8}{\@nx\enspace}\else
		\@ifempty{#8}{.}{\@nx\enspace}\fi
	}\@ifempty{#8}{\ifnum #2=\tw@ \def\@secnumfont{\bfseries}\fi}{}\protected@edef\@svsec{\ifnum#2<\@m
		\@ifundefined{#1name}{}{\ignorespaces\csname #1name\endcsname\space
		}\fi
		\@seccntformat{#1}}\fi
	\ifdim \@tempskipa>\z@ \begingroup #6\relax
	\@hangfrom{\hskip #3\relax\@svsec}{\interlinepenalty\@M #8\par}\endgroup
	\ifnum#2>\@m \else \@tocwrite{#1}{#8}\fi
	\else
	\def\@svsechd{#6\hskip #3\@svsec
		\@ifnotempty{#8}{\ignorespaces#8\unskip
			\@addpunct.}\ifnum#2>\@m \else \@tocwrite{#1}{#8}\fi
	}\fi
	\global\@nobreaktrue
	\@xsect{#5}}
\makeatother

\usepackage{geometry}
\def\rmlabel{\upshape({\kern-0.0833em\itshape\roman*\kern+0.0833em})}
\def\nlabel{\upshape(\kern-0.0833em{\itshape\arabic*}\kern0.0833em)}
\def\alabel{\upshape(\textit{\alph*})}
\def\Alabel{\upshape({\kern-0.0833em\itshape\Alph*\kern0.1111em})}

\let\setminus=\smallsetminus

\let\to=\lra

\makeatletter
\def\moverlay{\mathpalette\mov@rlay}
\def\mov@rlay#1#2{\leavevmode\vtop{   \baselineskip\z@skip \lineskiplimit-\maxdimen
		\ialign{\hfil$\m@th#1##$\hfil\cr#2\crcr}}}
\newcommand{\charfusion}[3][\mathord]{
	#1{\ifx#1\mathop\vphantom{#2}\fi
		\mathpalette\mov@rlay{#2\cr#3}
	}
	\ifx#1\mathop\expandafter\displaylimits\fi}
\makeatother

\newcommand{\dcup}{\charfusion[\mathbin]{\cup}{\cdot}}

\usepackage{caption}
\usepackage{subcaption}

\usepackage[mathlines]{lineno}
\usepackage{etoolbox} 

\newcommand*\linenomathpatch[1]{\expandafter\pretocmd\csname #1\endcsname {\linenomath}{}{}\expandafter\pretocmd\csname #1*\endcsname{\linenomath}{}{}\expandafter\apptocmd\csname end#1\endcsname {\endlinenomath}{}{}\expandafter\apptocmd\csname end#1*\endcsname{\endlinenomath}{}{}}
\newcommand*\linenomathpatchAMS[1]{\expandafter\pretocmd\csname #1\endcsname {\linenomathAMS}{}{}\expandafter\pretocmd\csname #1*\endcsname{\linenomathAMS}{}{}\expandafter\apptocmd\csname end#1\endcsname {\endlinenomath}{}{}\expandafter\apptocmd\csname end#1*\endcsname{\endlinenomath}{}{}}

\expandafter\ifx\linenomath\linenomathWithnumbers
\let\linenomathAMS\linenomathWithnumbers
\patchcmd\linenomathAMS{\advance\postdisplaypenalty\linenopenalty}{}{}{}
\else
\let\linenomathAMS\linenomathNonumbers
\fi

\linenomathpatchAMS{gather}
\linenomathpatchAMS{multline}
\linenomathpatchAMS{align}
\linenomathpatchAMS{alignat}
\linenomathpatchAMS{flalign}
\linenomathpatch{equation}

\usepackage{cleveref-usedon}

\theoremstyle{plain}
\newtheorem{theorem}{Theorem}[section]
\crefname{theorem}{Theorem}{Theorems}

\crefname{proposition}{Proposition}{Propositions}

\newtheorem{corollary}[theorem]{Corollary}
\crefname{corollary}{Corollary}{Corollaries}

\newtheorem{lemma}[theorem]{Lemma}
\crefname{lemma}{Lemma}{Lemmata}

\newtheorem{conjecture}[theorem]{Conjecture}
\crefname{conjecture}{Conjecture}{Conjectures}

\crefname{problem}{Problem}{Problem}

\newtheorem{claim}[theorem]{Claim}
\crefname{claim}{Claim}{Claims}

\crefname{observation}{Observation}{Observations}

\crefname{setup}{Setup}{Setups}

\newtheorem{fact}[theorem]{Fact}
\crefname{fact}{Fact}{Facts}

\crefname{algorithm}{Algorithm}{Algorithms}

\crefname{remark}{Remark}{Remarks}

\crefname{example}{Example}{Examples}

\theoremstyle{definition}
\newtheorem{definition}[theorem]{Definition}
\crefname{definition}{Definition}{Definitions}

\crefname{construction}{Construction}{Constructions}

\crefname{question}{Question}{Questions}

\numberwithin{equation}{section}

\crefname{section}{Section}{Sections}
\crefname{appendix}{Appendix}{Appendix}

\crefname{figure}{Figure}{Figures}

\def\COMMENT#1{}
\let\COMMENT=\footnote          

\usepackage{comment}

\let\polishlcross=\l
\def\l{\ifmmode\ell\else\polishlcross\fi}

\newcommand{\eps}{\varepsilon}
\renewcommand{\rho}{\varrho}
\newcommand{\sm}{\setminus}
\renewcommand{\subset}{\subseteq}

\DeclareMathOperator{\sh}{Sh}
\newcommand{\shij}{\sh_{i\rightarrow j}}
\newcommand{\shji}{\sh_{j\rightarrow i}} 

\let\th\relax
\DeclareMathOperator{\th}{\hbar}

\DeclareMathOperator{\crs}{cr}
\DeclareMathOperator{\eg}{{\normalfont\textsc{eg}}}

\def\tand{\ \text{and}\ }
\def\qand{\quad\text{and}\quad}

\newcommand{\cG}{\mathcal{G}}

\newcommand{\cM}{\mathcal{M}}

\DeclareMathSymbol{\meister}{\mathbin}{symbols}{"03}

\begin{document}
	
	\title{Tight Hamiltonicity from dense links of triples}
	
	\author[R.~Lang]{Richard Lang}
	\address {
		Fachbereich Mathematik,
		Universität Hamburg,
		Hamburg, Germany
	}
	
	\curraddr{Departament de Matemàtiques, Universitat Politècnica de Catalunya, Barcelona, Spain}
	\email{richard.lang@upc.edu}
	
	\author[M.~Schacht]{Mathias Schacht}
	\address {
		Fachbereich Mathematik,
		Universität Hamburg,
		Hamburg, Germany
	}
	\email{schacht@math.uni-hamburg.de}
	
	\author[J.~Volec]{Jan Volec}
	\address{
		Katedra teoretické informatiky, Fakulta informačních technologií,
		České vysoké učení technické v Praze,
		Prague, Czech Republic
	}
	\email{jan@ucw.cz}
	

	\begin{abstract}
		We show that for all $k\geq 4$, $\eps >0$, and $n$ sufficiently large, every $k$-uniform hypergraph on $n$ vertices in which each set of $k-3$ vertices is contained in at least $(5/8 + \eps) \binom{n}{3}$ edges contains a tight Hamilton cycle.
		This is asymptotically best possible.
	\end{abstract}
	
	
	\maketitle
	
	\vspace{-0.6cm}

	\section{Introduction}
	
	Our starting point is Dirac's theorem, which states that any graph $G$
	on $n \geq 3$ vertices and minimum degree $\delta(G) \geq n/2$
	contains a Hamilton cycle.
	Moreover, the constant~$1/2$ is best possible as exhibited by simple
	constructions.
	
	Over the past twenty-five years, this result has been extended to the hypergraph setting.
	Formally, a \emph{$k$-uniform hypergraph} (\emph{$k$-graph} for short)
	$G$ has a set of \emph{vertices}~$V(G)$ and a set of \emph{edges}~$E(G)$, where each edge consists of $k$ vertices, and we denote the
	number of edges~$|E(G)|$ by $e(G)$. For $1 \leq d \leq k-1$, the
	\emph{minimum $d$-degree of $G$}, denoted $\delta_d(G)$, is the maximum~$m$ such that every set of $d$ vertices is contained in at least $m$
	edges. A \emph{tight cycle} $C \subset G$ is a subgraph whose vertices
	are cyclically ordered such that every $k$ consecutive vertices form an
	edge. Moreover,~$C$ is \emph{Hamilton} if it spans all the vertices of $G$.
	We define the \emph{Dirac constant}~$\th_d^{(k)}$ as the (asymptotic)
	minimum $d$-degree threshold for tight Hamiltonicity. More precisely,
	$\th_d^{(k)}$ is the infimum $\varsigma \in [0,1]$ such that for every $\eps
	>0$ and $n$ sufficiently large, every $n$-vertex $k$-graph $G$ with
	$\delta_d(G) \geq (\varsigma + \eps)\binom{n-d}{k-d}$ contains a tight
	Hamilton cycle. So for example, Dirac's theorem implies that
	$\th_1^{(2)} = 1/2$.
	
	Minimum $d$-degree thresholds for tight Hamilton cycles were first
	investigated by Katona and Kierstead~\cite{KK99}, who observed that
	$\th_{k-1}^{(k)} \geq 1/2$ for all $k\geq 2$ and conjectured this to be
	tight (see~\cref{fig:constructions}).
	This conjecture was resolved by Rödl, Ruciński, and
	Szemerédi~\cites{RRS06,RRS08} by introducing the \emph{absorption
		method} in this setting. Since then the focus has shifted
	to degree types $d$ below $k-1$. After advances for nearly spanning cycles by Cooley and Mycroft~\cite{CM17},
	it was shown by Reiher, Rödl, Ruciński, Schacht, and
	Szemerédi~\cite{RRR19} that $\th_1^{(3)}=5/9$, which resolves the case
	of $d=k-2$ when $k=3$. Subsequently, this was generalised to
	$k=4$~\cite{PRRRSS20} and finally, Polcyn, Reiher, R\"odl, and
	Schülke~\cite{PRRS21} and, independently, Lang and
	Sanhueza-Matamala~\cite{LS22} established $\th_{k-2}^{(k)}=5/9$ for all
	$k \geq 3$.
	
	We focus on the case $d=k-3$.
	Han and Zhao~\cite{HZ16} provided a construction that
	shows $\th_{k-3}^{(k)} \geq 5/8$ (see \cref{fig:constructions}), which is believed to be optimal.
	Our main result confirms this conjecture.
	
	\begin{theorem}\label{thm:main}
		For every $k\geq 4$ and $\eps >0 $, there is $n_0$ such that every $k$-graph~$G$ on~$n\geq n_0$ vertices with $\delta_{k-3}(G) \geq (5/8 + \eps) \binom{n}{3}$ contains a tight Hamilton cycle.
	\end{theorem}
	
	In the following section, we give an outline of the argument and reduce
	\cref{thm:main} to two lemmata, whose proofs are given in
	\cref{sec:connection,sec:matching}. We conclude with a
	discussion and a few open problems in \cref{sec:conclusion}.

	\begin{figure}
		\tikzset{every picture/.style={line width=0.75pt}}
		\begin{tikzpicture}[x=0.75pt,y=0.75pt,yscale=-1,xscale=1]
			\draw   (30,74) .. controls (30,66.27) and (36.27,60) .. (44,60) -- (86,60) .. controls (93.73,60) and (100,66.27) .. (100,74) -- (100,175.5) .. controls (100,183.23) and (93.73,189.5) .. (86,189.5) -- (44,189.5) .. controls (36.27,189.5) and (30,183.23) .. (30,175.5) -- cycle ;
			\draw   (110.25,74.25) .. controls (110.25,66.52) and (116.52,60.25) .. (124.25,60.25) -- (166.25,60.25) .. controls (173.98,60.25) and (180.25,66.52) .. (180.25,74.25) -- (180.25,175.75) .. controls (180.25,183.48) and (173.98,189.75) .. (166.25,189.75) -- (124.25,189.75) .. controls (116.52,189.75) and (110.25,183.48) .. (110.25,175.75) -- cycle ;
			\draw  [color={rgb, 255:red, 208; green, 2; blue, 27 }  ,draw opacity=1 ][fill={rgb, 255:red, 208; green, 2; blue, 27 }  ,fill opacity=0.15 ] (70,160) .. controls (70,148.95) and (85.67,140) .. (105,140) .. controls (124.33,140) and (140,148.95) .. (140,160) .. controls (140,171.05) and (124.33,180) .. (105,180) .. controls (85.67,180) and (70,171.05) .. (70,160) -- cycle ;
			\draw  [color={rgb, 255:red, 208; green, 2; blue, 27 }  ,draw opacity=1 ][fill={rgb, 255:red, 208; green, 2; blue, 27 }  ,fill opacity=0.15 ] (60,139.86) .. controls (48.96,139.78) and (40.11,124.05) .. (40.25,104.72) .. controls (40.38,85.39) and (49.44,69.78) .. (60.49,69.86) .. controls (71.53,69.94) and (80.38,85.67) .. (80.24,105) .. controls (80.11,124.33) and (71.05,139.94) .. (60,139.86) -- cycle ;
			\draw  [color={rgb, 255:red, 74; green, 144; blue, 226 }  ,draw opacity=1 ][fill={rgb, 255:red, 74; green, 144; blue, 226 }  ,fill opacity=0.25 ][dash pattern={on 4.5pt off 4.5pt}] (150.77,139.78) .. controls (139.72,139.9) and (130.6,124.33) .. (130.38,105) .. controls (130.17,85.67) and (138.95,69.9) .. (150,69.78) .. controls (161.04,69.66) and (170.17,85.23) .. (170.38,104.56) .. controls (170.59,123.89) and (161.81,139.65) .. (150.77,139.78) -- cycle ;
			\draw   (220,74.25) .. controls (220,66.52) and (226.27,60.25) .. (234,60.25) -- (276,60.25) .. controls (283.73,60.25) and (290,66.52) .. (290,74.25) -- (290,175.75) .. controls (290,183.48) and (283.73,189.75) .. (276,189.75) -- (234,189.75) .. controls (226.27,189.75) and (220,183.48) .. (220,175.75) -- cycle ;
			\draw   (300.25,74.5) .. controls (300.25,66.77) and (306.52,60.5) .. (314.25,60.5) -- (356.25,60.5) .. controls (363.98,60.5) and (370.25,66.77) .. (370.25,74.5) -- (370.25,176) .. controls (370.25,183.73) and (363.98,190) .. (356.25,190) -- (314.25,190) .. controls (306.52,190) and (300.25,183.73) .. (300.25,176) -- cycle ;
			\draw  [color={rgb, 255:red, 208; green, 2; blue, 27 }  ,draw opacity=1 ][fill={rgb, 255:red, 208; green, 2; blue, 27 }  ,fill opacity=0.15 ] (260,170.13) .. controls (260,164.53) and (275.67,160) .. (295,160) .. controls (314.33,160) and (330,164.53) .. (330,170.13) .. controls (330,175.72) and (314.33,180.25) .. (295,180.25) .. controls (275.67,180.25) and (260,175.72) .. (260,170.13) -- cycle ;
			\draw  [color={rgb, 255:red, 208; green, 2; blue, 27 }  ,draw opacity=1 ][fill={rgb, 255:red, 208; green, 2; blue, 27 }  ,fill opacity=0.15 ] (250.07,130) .. controls (239.03,129.92) and (230.17,116.45) .. (230.28,99.92) .. controls (230.4,83.38) and (239.44,70.03) .. (250.49,70.11) .. controls (261.53,70.19) and (270.4,83.66) .. (270.28,100.2) .. controls (270.16,116.73) and (261.12,130.08) .. (250.07,130) -- cycle ;
			\draw  [color={rgb, 255:red, 74; green, 144; blue, 226 }  ,draw opacity=1 ][fill={rgb, 255:red, 74; green, 144; blue, 226 }  ,fill opacity=0.15 ][dash pattern={on 4.5pt off 4.5pt}] (340.33,130.22) .. controls (329.47,130.34) and (320.51,117) .. (320.33,100.44) .. controls (320.15,83.87) and (328.81,70.34) .. (339.67,70.22) .. controls (350.53,70.1) and (359.49,83.44) .. (359.67,100) .. controls (359.85,116.57) and (351.19,130.1) .. (340.33,130.22) -- cycle ;
			\draw  [color={rgb, 255:red, 74; green, 144; blue, 226 }  ,draw opacity=1 ][fill={rgb, 255:red, 74; green, 144; blue, 226 }  ,fill opacity=0.15 ][dash pattern={on 4.5pt off 4.5pt}] (260,140.13) .. controls (260,134.53) and (275.67,130) .. (295,130) .. controls (314.33,130) and (330,134.53) .. (330,140.13) .. controls (330,145.72) and (314.33,150.25) .. (295,150.25) .. controls (275.67,150.25) and (260,145.72) .. (260,140.13) -- cycle ;
			
			\draw (60.24,104.86) node  [font=\normalsize,color={rgb, 255:red, 0; green, 0; blue, 0 }  ,opacity=1 ] [align=left] {3};
			\draw (83,153) node [anchor=north west][inner sep=0.75pt]  [color={rgb, 255:red, 0; green, 0; blue, 0 }  ,opacity=1 ] [align=left] {2};
			\draw (117,153) node [anchor=north west][inner sep=0.75pt]  [color={rgb, 255:red, 0; green, 0; blue, 0 }  ,opacity=1 ] [align=left] {1};
			\draw (150.38,105.78) node  [font=\normalsize,color={rgb, 255:red, 0; green, 0; blue, 0 }  ,opacity=1 ] [align=left] {3};
			\draw (250.5,101) node  [font=\normalsize,color={rgb, 255:red, 0; green, 0; blue, 0 }  ,opacity=1 ] [align=left] {4};
			\draw (275,163) node [anchor=north west][inner sep=0.75pt]  [color={rgb, 255:red, 0; green, 0; blue, 0 }  ,opacity=1 ] [align=left] {3};
			\draw (305,163) node [anchor=north west][inner sep=0.75pt]  [color={rgb, 255:red, 0; green, 0; blue, 0 }  ,opacity=1 ] [align=left] {1};
			\draw (340.38,101.03) node  [font=\normalsize,color={rgb, 255:red, 0; green, 0; blue, 0 }  ,opacity=1 ] [align=left] {4};
			\draw (275,133) node [anchor=north west][inner sep=0.75pt]  [color={rgb, 255:red, 0; green, 0; blue, 0 }  ,opacity=1 ] [align=left] {1};
			\draw (305,133) node [anchor=north west][inner sep=0.75pt]  [color={rgb, 255:red, 0; green, 0; blue, 0 }  ,opacity=1 ] [align=left] {3};

		\end{tikzpicture}
		
		\caption{The picture shows a $3$-graph on the left and a $4$-graph on the right.
			Both hypergraphs have their vertex sets partitioned into two parts of equal size.
			The drawn edges indicate that the graphs contain all edges of this type.
			The colours highlight the respective tight components.
			Since any tight cycle is either red or blue, neither of the hypergraphs admits a tight Hamilton cycle.
			This gives a lower bound for the corresponding minimum degree thresholds. Indeed, the $3$-graph has a
				minimum $2$-degree close to $\tfrac{1}{2}n$, while the $4$-graph has minimum $1$-degree close to $\tfrac{5}{8}\binom{n}{3}$.
		}
		\label{fig:constructions}
	\end{figure}

	\section{From Dirac to Erdős--Gallai}
	
	Let us recall two classic problems from extremal combinatorics. The
	first one is Erdős' Matching Conjecture~\cite{Er65}, which predicts the
	size of a largest matching we are guaranteed to find in a $k$-graph
	with a given number of vertices and edges.
	A matching is a subgraph with pairwise disjoint edges.
	
	\begin{conjecture}\label{con:EMC}
		Let $G$ be a $k$-graph on $n$ vertices that does not contain a matching of more than $s$ edges.
		Then
		\begin{align*}
			e(G) \leq \max\left\{\binom{(s+1)k-1}{k},\, \binom{n}{k} - \binom{n-s}{k}\right\}.
		\end{align*}
	\end{conjecture}
	
	The conjecture was resolved
	by Erdős and Gallai~\cite{EG59} for $k=2$ and for $k=3$ by \L uczak and
	Mieczkowska~\cite{LM14} (when $n$ is large) and Frankl~\cite{Fra17} (for all $n$). The second problem
	concerns the Erdős--Gallai Theorem~\cite{EG59}, which tells us the size
	of a longest cycle in a graph of given density. This problem has a
	natural extension to tight cycles in hypergraphs, and it was first studied by
	Gy\H{o}ri, Katona, and Lemons~\cite{GKL16} and Allen, Böttcher, Cooley, and Mycroft~\cite{ABCM17}.
	
	How are these questions related to Dirac-type problems?
	The idea is to study the internal structure of the neighbourhoods.
	Alon, Frankl, Huang, Rödl, Ruciński, and Sudakov~\cite{AFH+12}
	showed that the problem of determining the minimum $d$-degree threshold
	for $k$-uniform perfect matchings can be reduced to a special case of
	Erdős' Matching Conjecture for $(k-d)$-graphs, applied to
	the $(k-d)$-uniform link hypergraphs. Analogously, it was shown by Lang
	and Sanhueza-Matamala~\cite{LS22} that the problem of determining the minimum
	$d$-degree threshold for $k$-uniform tight Hamilton cycles can be
	reduced to an Erdős--Gallai-type question for $(k-d)$-uniform link hypergraphs.
	(A simpler proof of this result can be found in the more recent work of the same authors~\cite{LS24a}.)
	
	To formalise this discussion, we require some additional terminology.
	For a $k$-graph $G$, let~$G^{\meister}$ be
	the \emph{line graph} on the vertex
	set~$E(G)$ with an edge $ef$ whenever $|e \cap f| = k-1$. A subgraph
	$H\subset G$ is said to be
	\emph{tightly connected} if it has no isolated vertices and~$E(H)$ induces a connected subgraph in $G^\meister$.
	Moreover, we refer to edge maximal tightly connected subgraphs as
	\emph{tight components}.
	
	\begin{theorem}[{Lang and Sanhueza-Matamala~\cite{LS22}*{Theorem~11.5}}]\label{thm:LS}
		Suppose for every $\eps >0$, there are $\gamma>0$ and $n_0$ such that every $3$-graph $G$ on $n\geq n_0$ vertices with $e(G) \geq (5/8 + \eps) \binom{n}{3}$ contains a subgraph $C \subset G$ such that
		\begin{enumerate}[label=\rmlabel]
			\item \label{itm:connectivity} $C$ is tightly connected,
			\item \label{itm:density} $e(C) \geq (1/2+\gamma) \binom{n}{3}$, and
			\item \label{itm:space} $C$ has a matching of size at least $(1/4+\gamma)n$.
		\end{enumerate}
		Then $\th_{k-3}^{(k)}=5/8$ for every $k\geq 4$.\qed
	\end{theorem}
	
	Against this backdrop, our argument proceeds as follows. Let $G$ be an
	$n$-vertex $3$-graph with $e(G) \geq (5/8 + \eps) \binom{n}{3}$. In a first step, we establish the existence of
	a tight component $C \subset G$ that
	satisfies conditions~\ref{itm:connectivity} and~\ref{itm:density} of \cref{thm:LS}.
	
	\begin{lemma}[Connection]\label{lem:connection}
		For all $\eps>0$, there exists $n_0$ such that every $3$-graph~$G$ on $n\geq n_0$ vertices with $e(G) \geq (5/8 + \eps) \binom{n}{3}$ contains a tight component $C \subset G$ with $e(C) \geq (1/2+\eps) \binom{n}{3}$.
	\end{lemma}
	
	We then show that any such tight component $C$ provided by \cref{lem:connection} contains a large matching satisfying condition \ref{itm:space} of \cref{thm:LS}.
	
	\begin{lemma}[Matching]\label{lem:matching}
		For all $\eps >0$, there exist $\gamma>0$ and $n_0$ such that every $3$-graph~$G$ on $n\geq n_0$ vertices with $e(G) \geq (5/8 + \eps) \binom{n}{3}$ and a tight component $C \subset G$ with $e(C) \geq (1/2+\eps) \binom{n}{3}$ contains a matching $M \subset C$ of size at least $(1/4+\gamma)n$.
	\end{lemma}
	The proof of \cref{lem:connection} can be found in the next section. It is an easy consequence of a
	Kruskal--Katona-type result obtained, independently, by Frankl, Kato, Katona, and Tokushige~\cite{FKKT13}
	as well as by Huang, Linial, Naves, Peled, and Sudakov~\cite{HLN+16}.
	For the proof of \cref{lem:matching}, presented in \cref{sec:matching}, we adapt a strategy of
	Łuczak and Mieczkowska~\cite{LM14} from their proof of  Erdős' Matching
	Conjecture for $3$-graphs.

	\section{The connection lemma}\label{sec:connection}
	
	We start by stating the aforementioned Kruskal--Katona-type result~\cites{FKKT13,HLN+16}.
	We denote the complete graph on $n$ vertices by $K_n$.
	
	\begin{theorem}\label{lem:triangles}
		For every $\eps>0$, there exists $n_0$ such that for every $n\geq n_0$ the following holds.
		Suppose that the edges of a subgraph of $K_n$ are coloured red and blue
		such that there are at least $\binom{n}{3}/8$ monochromatic triangles in each colour.
		Then there are fewer than $(5/8 + \eps) \binom{n}{3}$  monochromatic triangles all together. \qed
	\end{theorem}
	
	Below we derive \cref{lem:connection} from \cref{lem:triangles}.
	
	\begin{proof}[Proof of \cref{lem:connection}]
		For given $\eps > 0$, we set $n_0$ according to \cref{lem:triangles}.
		Let $G$ be a $3$-graph on $n\geq n_0$ vertices with $e(G) \geq  (5/8 + \eps) \binom{n}{3}$.
		Without loss of generality, we may assume~$\eps<1/16$ and $e(G) \leq (5/8+2\eps) \binom{n}{3}$.
		
		Let $C_1,\dots,C_\l$ be the tight components of $G$ with $e(C_1)\geq \dots \geq e(C_\l)$.
		Aiming for a contradiction,
		we impose that $e(C_1)<(1/2+\eps)\binom{n}{3}$.
		We shall group the tight components of~$G$ into a partition $R\dcup B=E(G)$ such that
		\begin{equation}\label{eq:triangles1}
			\max\big\{|B|,|R|\big\}
			<
			\Big(\frac{1}{2} + \eps\Big) \binom{n}{3}
		\end{equation}
		and no edge of $R$ intersects with any edge of $B$ in more than one
		vertex. Indeed, considering the smallest integer $i$ such that
		\[
		e(C_1)+\dots+e(C_i)\geq \Big(\frac{1}{2} + \eps\Big) \binom{n}{3}
		\]
		and recalling the upper bound on $e(G)$, tells us
		\[
		e(C_{i+1})+\dots+e(C_\l) \leq \Big(\frac{1}{8} + \eps\Big) \binom{n}{3}\,.
		\]
		Note that $e(C_i)< (3/8)\binom{n}{3}$, since otherwise the monotonicity of $e(C_j)$
		combined with the upper bound $e(G) \leq (5/8+2\eps) \binom{n}{3}$ implies $i=1$, which contradicts the choice of $i$.
		Consequently,  the sets $R=E(C_1)\cup\dots\cup E(C_{i-1})$ and
		$B=E(C_i)\cup\dots\cup E(C_{\l})$ form the desired partition $R\dcup B=E(G)$.
		
		Now observe that inequality~\eqref{eq:triangles1} and the given lower bound $e(G) \geq  (5/8 + \eps) \binom{n}{3}$ yields
		\[
		\frac{1}{8}\binom{n}{3}
		\leq
		\min\big\{|B|,|R|\big\}\,.
		\]
		From this, we derive an edge-colouring of
		$K_n$ by giving an edge colour red if it is contained in a triple
		of $R$ and colour blue if it is contained in a triple of $B$.
		(The remaining edges are coloured arbitrarily.)
		By
		\cref{lem:triangles} it follows that there are fewer than $(5/8 + \eps)
		\binom{n}{3}$ monochromatic triangles. But this contradicts the assumption
		that $e(G) \geq (5/8 + \eps) \binom{n}{3}$.
	\end{proof}

	\section{The matching lemma}\label{sec:matching}
	
	In this section, we establish \cref{lem:matching}. We proceed by studying the extremal function for matchings
	hosted by a largest tight component in a $3$-graph $G$.
	If $G$ has edge density above~$5/8$, then we are guaranteed by \cref{lem:connection} a tight component $C$ of density at least~$1/2$.
	Applying Erdős' Matching Conjecture (\cref{con:EMC}) to $C$ gives a matching of size at least $(1-2^{-1/3})n \approx 0.206 n$, which
	does not suffice for the aspirations set in \cref{lem:matching}.
	However, this na\"\i ve approach turns out to be suboptimal, because there might be edges in $G$ that lie outside of $C$,
	which contribute indirectly to the size of a largest matching by obstructing the space.
	The extremal constructions arising from these restrictions are no longer captured by Erdős' Matching Conjecture, and
	we thus require a more nuanced analysis to deduce \cref{lem:matching}.
	
	In \cref{sec:extc} we formulate the corresponding extremal problem, and we reduce \cref{lem:matching}
	to it (see \cref{lem:EG-5/8}). The rest of \cref{sec:matching} is devoted to the proof of \cref{lem:EG-5/8}.
	The proof is based on the approach of \L uczak and Mieczkowska~\cite{LM14} to show the Erdős' Matching Conjecture
	for $3$-graphs. In particular, we also employ the \emph{shifting technique},
	which is described in \cref{sec:shift}.

	\subsection{Extremal function for matchings in tight components}
	
	\label{sec:extc}
	We write $m(R)$ for the size of a largest matching in a $3$-graph $R$.
	We call $3$-graphs $R$ and $B$ \emph{distinguishable} if the edges of
	$R$ and $B$ only intersect in single vertices.
	So in particular, distinct tight components are distinguishable.
	\begin{definition}
		We define $\cG(n,s,t)$ as
		the family of all pairs $(R,B)$ such that $R$ and~$B$ are distinguishable $3$-graphs on the vertex set $\{1,\dots,n\}$ with $m(R)\leq s$ and $e(R) > t$.
	\end{definition}
	Moreover, we define the extremal number
	\[
	\mu(n,s,t)
	=
	\max \big\{e(R \cup B) \colon (R,B) \in \cG(n,s,t)\big\}\,,
	\]
	and we denote the family of extremal pairs by
	\[
	\cM(n,s,t)
	=
	\big\{(R,B)\in\cG(n,s,t)\colon e(R \cup B) = \mu(n,s,t)\big\}\,.
	\]
	In view of \cref{lem:matching}, we are interested
	in~$\mu(n,s,t)$ for $s\approx n/4$ and $t\approx\binom{n}{3}/2$, which is rendered by the main lemma of this section.
	\begin{lemma}\label{lem:EG-5/8}
		For each $\gamma>0$, we have
		$\mu\big(n,n/4,\tbinom{n}{3}/2\big) \leq ( {5}/{8}+\gamma )  \binom{n}{3}$
		for sufficiently large~$n$.
	\end{lemma}
	Below we deduce \cref{lem:matching} as a simple consequence of \cref{lem:EG-5/8}
	and, consequently, for the proof of \cref{thm:main} it then only remains to establish \cref{lem:EG-5/8}.
	
	\begin{proof}[Proof of \cref{lem:matching}]
		For a given $\eps >0 $, we set $\gamma=\eps/11$ and let $n_0$ be sufficiently large.
		Given~$G$ and a tight component~$C$ satisfying the assumptions of \cref{lem:matching}, we first fix an arbitrary matching
		$M' \subseteq C$ of size $\gamma n$.
		
		Set $R = C - V(M')$, and let $B$ be obtained from $G-V(M')$ by deleting all edges of~$R$.
		After adding isolated vertices, we can assume that $R$ and $B$ are $3$-graphs on the vertex set~$\{1,\dots,n\}$
		with $e(R) \geq (1/2+\eps/2) \binom{n}{3}$ and
		\[
		e(R \cup B)
		\geq
		e(G) - 3\gamma n\cdot \binom{n}{2}
		\geq
		\Big(\frac{5}{8} +\eps-10\gamma\Big) \binom{n}{3}
		=
		\Big(\frac{5}{8} +\gamma\Big) \binom{n}{3}
		\]
		for sufficiently large~$n$.
		Moreover, $R$ and $B$ are distinguishable, since $C$ is a tight component.
		It follows that $R$ contains a matching $M''$ of size $n/4$ by \cref{lem:EG-5/8}
		and $M=M' \cup M''$ is the desired matching in $C$.
	\end{proof}

	\subsection{Shifting}\label{sec:shift}
	
	For a $k$-graph $G$ with $i$, $j \in V(G)$, the \emph{$(i,j)$-shift} of
	$G$, denoted by $\sh_{i \rightarrow j}(G)$, is obtained from $G$ by
	replacing each edge $e \in E(G)$ with
	\begin{align*}
		f = (e\setminus\{i\}) \cup \{j\}\
		\text{provided that}\
		i \in e\,,\
		j \notin e\tand
		f  \notin E(G)\,.
	\end{align*}
	We mainly consider $3$-graphs here. However, we will also study the \emph{shadow} $\partial G$ of a
	$3$-graph~$G$, defined as the $2$-graph on $V(G)$ containing an edge $e \in E(\partial G)$ whenever there is a
	triple $f \in E(G)$ with  $e \subset f$. In the proof of \cref{lem:EG-5/8}, we shall
	use the relation of shifted $3$-graphs and its shadow, and for that we defined the $(i,j)$-shift for $k$-graphs in general.
	Below we collect the basic facts about shifted $3$-graphs for our proof. For a
	more comprehensive review we refer to the survey of Frankl~\cite{Fra87} and the
	monograph of Frankl and Tokushige~\cite{FranklTokushige2018}.
	
	\begin{lemma}\label{lem:shift}
		For every $3$-graph $G$ with $i$, $j \in V(G)$ and distinguishable subgraphs $R$, $B\subseteq G$,
		the following holds:
		\begin{enumerate}[label=\alabel]
			\item \label{it:shift-a} $e(\shij(G))=e(G)$,
			\item \label{it:shift-b} $m(\shij(G))\leq m(G)$, and
			\item \label{it:shift-c} $\shji(R)$ and $\shij(B)$ are distinguishable.
		\end{enumerate}
	\end{lemma}
	\begin{proof}
		Assertion~\ref{it:shift-a} follows from the definition of an $(i,j)$-shift and
		property~\ref{it:shift-b} is a well known fact~\cite{LM14}*{Lemma~3}.
		
		For the proof of~\ref{it:shift-c}, we assume by contradiction that there are edges $e\in E(\shji(R))$ and
		$f\in E(\shij(B))$ with $|e\cap f|\geq 2$.
		Since $R$ and $B$ are distinguishable, we may assume by symmetry that $e \notin E(R)$, writing $e=uvi$ and $e'=uvj\in E(R)$.
		
		If $ uv \subset f$, then there would be an edge $uvw\in E(B)$ for some $w\in V(G)$,
		which contradicts the distinguishedness of $R$ and $B$. Consequently, without loss of generality we have
		$f=u'vi$ for some $u'\neq u$. If $u'=j$, then $f$ is also an edge of $B$, and $f$ and~$e'$ contradict
		that $R$ and $B$ are distinguishable. In the remaining case $u'\neq j$, we arrive at $u'vj\in E(B)$, which
		again contradicts the assumed distinguishability.
	\end{proof}
	
	As usual we shall study $3$-graphs (and their shadow), which are fully shifted in one direction.
	More precisely, we say a $k$-graph $G$ on the vertex set $\{1,\dots,n\}$ is \emph{left-shifted}
	if~$\shji(G)=G$ for all $i<j$ and, similarly, it is \emph{right-shifted} if $\shij(G)=G$ for all~$i<j$.
	It is easy to see that we can obtain a left-shifted $k$-graph from any
	given~$G$ after a finite sequence of $(j,i)$-shifts with $i<j$. Combining this fact with \cref{lem:shift}
	tells us that there are shifted extremal examples in $\cM(n,s,t)$.
	\begin{corollary}\label{cor:shifted-R-B}
		For all integers $n$, $s$, and $t$ there is a pair $(R,B) \in \cM(n,s,t)$
		such that $R$ and $\partial R $ are left-shifted, while $B$ and $\partial B$ are right-shifted.
	\end{corollary}
	\begin{proof}
		Consider an arbitrary pair $(R',B')\in\cM(n,s)$ and let $1\leq i<j\leq n$. It follows from
		\cref{lem:shift}\,\ref{it:shift-a}--\ref{it:shift-c} that $\big(\shji(R'),\shij(B')\big)$ is also
		in $\cM(n,s,t)$. Since the degree of the vertex $j$ in $\shji(R')$ is smaller than in $R'$ if
		$R'\neq  \shji(R')$ and, similarly, the degree increases in $\shij(B')$ if $B'\neq  \shij(B')$,
		it follows that after a finite sequence of such simultaneous~$(j,i)$- and $(i,j)$-shifts in the respective subgraphs,
		we arrive at a pair $(R,B)\in \cM(n,s,t)$ with $R$ being left-shifted and $B$ being right-shifted.
		
		Finally, let us argue that $\partial R$ is left-shifted.
		To this end, consider $i<j$ and $vj\in E(\partial R)$. Consequently, there is some vertex $u$ such that
		$uvj\in E(R)$. Since $R$ is left-shifted, we have either have $u=i$ or $uvi\in E(R)$, but in both cases
		we arrive at $vi\in E(\partial R)$, and this shows that $\partial R$ is left-shifted indeed.
		The argument for $\partial B$ follows analogously.
	\end{proof}

	\subsection{Local structure}
	
	In the proof of \cref{lem:matching}, we reduce the initial problem to a optimisation question about local structure in a maximal matching.
	The assumptions of the following lemma correspond to the structural conditions identified over the course of the proof of \cref{lem:matching} and are then processed into the correct quantitative answer.
	
	\begin{lemma}\label{lem:local-structure}
		Let $T = \{M_1 , M_2 , M_3\}$ consists of three pairwise disjoint triples $M_\ell = (i_\ell,j_\ell,k_\ell)$ of integers such that $i_\ell < j_\ell < k_\ell$ for $1\leq \ell \leq 3$.
		Let $R$ and $B$ be non-uniform, partite hypergraphs on $M_1\dcup M_2\dcup M_3$ with singletons $R_1$ and $B_1$, pairs $R_2$ and $B_2$ as well as triples~$R_3$ and $B_3$.
		Suppose $R$ and $B$ satisfy the following properties:
		\begin{enumerate}[label=\alabel]
			\item \label{itm:shadow} The hypergraph $B$ is down-closed.
			
			\item \label{itm:distinguishable} The edges of  $R_3$ and $B_3$ are distinguishable, and $R_2$ and $B_2$ are disjoint.
			
			\item\label{itm:expanding} There are no three disjoint edges in $T \cup R$ whose union intersects $\{i_1,i_2,i_3\}$ in at most two elements.
			
			\item \label{itm:crossing}
			The edges of $R$ and $B$ are crossing with respect to $T$.
			
			\item \label{itm:shifted}
			The graph $R_2$ is left-shifted in the crossing sense, that is,
			if $x' \leq x$ and $y' \leq y$ are two pairs of vertices from different triples of $T$ with $xy \in R_2$, then $x'y' \in R_2$.
			Similarly, $B_2$ are right-shifted (in the analogous way).
			
			\item \label{itm:weights}
			Suppose $r\colon R_1 \cup R_2 \to [0,1]$, $b\colon B_1 \cup B_2 \to [0,1]$
			satisfy
			\begin{align*}
				r(u)  + b(v)  & \leq 1\ \text{for all $u\in R_1$, $v\in B_1$}\,     \\
				r(vw) + b(v)  & \leq 1\ \text{for all $vw\in R_2$, $v\in B_1$}\,    \\
				r(u)  + b(uw) & \leq 1\ \text{for all $u\in R_1$, $uw\in B_2$, and} \\
				r(vw) + b(uw) & \leq 1\ \text{for all  $vw\in R_2$, $uw\in B_2$}\,.
			\end{align*}
			{Moreover, $b(i_\ell) \leq b(j_\ell) \leq b(k_\ell)$ for every $1 \leq \ell \leq 3$.}
			
			\item \label{itm:blue-star}  For distinct $1\leq p, \ell \leq 3$ with $k_p > k_\ell$, the pairs of $B_2$ between $M_\ell$ and $M_p$ form a star centred at $k_p$.
			
			\item \label{itm:red-pairs} The singletons of $R_1$ are among $i_1$, $i_2$, and $i_3$.
			For every $1\leq p < \ell \leq 3$, there are at most $5$ pairs of $R_2$ between $M_\ell$ and $M_p$ with equality if and only if the pairs are $i_pi_\ell$, $i_pj_\ell$, $i_pk_\ell$,  $j_pi_\ell$, and~$k_pi_\ell$.
			
			\item \label{itm:many-red-triples} If $|R_3| \geq 20$, then there are at most $4$ pairs of $R_2$ between any two triples of $T$.
			{Moreover, if $|R_3|\geq 22$ and $R_2$ contains
				four pairs between any two triples of $T$, then none of those pairs
				is incident to $k_1$, $k_2$, or $k_3$.}
		\end{enumerate}
		Finally, suppose reals $\sigma$, $\beta$, $q_1$, and $q_2$ satisfy
		\[
		1-2^{-3}\leq \sigma \leq \frac{1}{4}\,,\qquad
		\beta = \frac{1}{\sigma}-3\,,
		\]
		and
		\begin{align}
			q_1   = \sum_{u \in R_1} r(u)  + \sum_{v \in B_1}b(v)
			\qand
			q_2 = \sum_{vw \in R_2} r(vw)  + \sum_{uw \in B_2}b(uw)\,.
			\label{eq:defqs}
		\end{align}
		Then $\sigma^3\left(\beta^3 + \beta^2 q_1 + \beta q_2 + e(R_3 \cup B_3)\right) \leq 5/8.$
	\end{lemma}

	\subsection{Proof of Lemma~\ref{lem:EG-5/8}}
	
	For a (not necessarily uniform) hypergraph $H$ and subsets $S$, $W \subset
	V(G)$, we denote by $\deg_H(S;W)$ the number of edges $S \cup Y$ in $H$
	with $Y \subset W$. To emphasise (or specify) the uniformity of an edge
	in a hypergraph, we sometimes speak of an (unordered) triple, pair or
	singleton.
	
	\begin{proof}[Proof of \cref{lem:EG-5/8}]
		Given $\gamma > 0$, let $n$ be sufficiently large and
		consider a pair
		\begin{align*}
			(R,B) \in \cM(n,n/4,\tbinom{n}{3}/2)\,.
		\end{align*}
		So in particular, $R$ and $B$ are distinguishable $3$-graphs on the vertex set $[n]=\{1,\dots,n\}$, a largest matching in $R$ has size at most $n/4$ and $e(R) \geq \binom{n}{2}/2$.
		Moreover, $e(R \cup B)$  is maximal subject to these constraints.
		We have to show that $e(R\cup B) \leq (5/8 + \gamma) \binom{n}{3}$.
		
		By \cref{cor:shifted-R-B} we can assume that $R$,
		$\partial R$ are left-shifted, while~$B$, $\partial B$ are
		right-shifted. We refer to the edges of $R$, $\partial R$ and
		$B$, $\partial B$ as \emph{red} and \emph{blue}, respectively.
		Let $M = \{(i_\ell,j_\ell,k_\ell) \colon 1 \leq \ell \leq s\}$ be a
		largest matching in $R$ with $i_\ell < j_\ell < k_\ell$ for each
		$\ell$.
		By the resolution of Erdős' Matching Conjecture (\cref{con:EMC}) for $3$-graphs~\cites{LM14,Fra17}, we can assume that $s = \sigma n$ for a real $\sigma$ satisfying
		\[
		\big(1-2^{-1/3}\big)
		\leq
		\sigma
		\leq \frac{1}{4}\,.
		\]
		We partition the vertex set of $M$ into three parts
		\begin{align*}
			V(M) = I \cup J \cup K
		\end{align*}
		such that for every edge $(i,j,k) \in M$, we have $i \in I$,
		$j \in J$, and $k \in K$.
		A set of vertices is called \emph{crossing} if it contains at most one
		vertex of every matching edge of $M$.
		(This includes singletons, naturally.)
		
		The remainder of the argument proceeds by analysing the local
		configurations of the matching $M$.
		Indeed, double counting allows us to bound the edges in $R \cup B$ by focusing on a typical triple of edges from $M$. For such a triple we shall verify the assumptions of \cref{lem:local-structure}, which then provides an upper
		bound to the number of involved edges in $R\cup B$.
		
		To formalise the double counting argument, denote by $W$ the set
		of vertices that are not covered by $M$. Obviously, none of the edges
		of $R$ are contained in $W$ by maximality of $M$.
		Now consider a triple $T$ of matching edges from $M$.
		We denote the crossing subsets of~$V(T)$ by $\crs(T)$.
		Observe that every singleton of $V(M)$ appears in exactly~$\tbinom{|M|-1}{2}$ such triples, while every crossing pair of $V(M)$ appears in exactly $|M|-2$ such triples.
		Moreover, every crossing triple appears of course only in a single triple.
		This leads to the following definitions
		\begin{align*}
			e_1 (T) & = \frac{1}{\binom{|M|-1}{2}} \sum_{u \in  \crs(T)}    \deg_R(u;W) + \deg_B(u;W)\,, \\
			e_2(T)  & = \frac{1}{|M|-2} \sum_{uv \in \crs(T)}    \deg_R(uv;W) + \deg_B(uv;W)\,,          \\
			e_3(T)  & =  |  \crs(T) \cap (R \cup B) |\,.
		\end{align*}
		
		Given this setup, we double count the edges of $R \cup B$ in $V(M)$ along the
		triples~$T$ of matching edges from $M$ to obtain
		\begin{align*}
			e(R \cup B)
			& \leq
			\sum_{T\in \binom{M}{3}} \big(e_1(T)  + e_2(T) + e_3(T)\big) + \binom{|W|}{3}+\frac{\gamma}{8}\binom{n}{3}
			\,,
		\end{align*}
		where the term $\binom{|W|}{3}$ accounts for the blue edges in $W$ and
		$\frac{\gamma}{8}\binom{n}{3}$ bounds the number of remaining non-crossing edges.
		Recalling that $|M| = \sigma n$ and $|W| = (1-3\sigma) n$, we then rewrite and bound the right-hand side as follows
		\begin{align}
			e(R \cup B)
			& \leq
			\sum_{T\in \binom{M}{3}}\bigg(e_1(T)  + e_2(T) + e_3(T) + \frac{\binom{|W|}{3}}{\binom{|M|}{3}}\bigg)
			+\frac{\gamma}{8}\binom{n}{3}\nonumber \\
			& <
			\sum_{T\in \binom{M}{3}}\bigg(e_1(T)  + e_2(T) + e_3(T) + \Big(\frac{1}{\sigma}-3\Big)^3\bigg) +  \frac{\gamma}{4} \binom{n}{3}\,. \label{equ:global-bound-on-R-cup-B}
		\end{align}
		
		For the remainder, we aim
		to suitably bound
		$e_1(T) + e_2(T) + e_3(T) + (1/\sigma-3)^3$
		for most
		triples~$T$ of matching edges in $M$.
		
		In our analysis, we shall focus on crossing singletons and pairs that reside in many red triples intersecting $W$.
		For that, we define an auxiliary (non-uniform) hypergraph $H$ on~$V(M)$ with edge set $M \cup R_1 \cup R_2 \cup R_3$, where
		\begin{align*}
			R_1 & = \{ u  \colon \deg_{R}(u;W) \geq 20n  \}\,,                               \\
			R_2 & = \{ uv   \colon \text{$uv$ is crossing and $\deg_R(uv;W) \geq 20$}  \}\,, \\
			R_3 & = \{ uvw   \colon \text{$uvw$ is a {crossing} edge in $R$}\}\,.
		\end{align*}
		Note that since $R$ is left-shifted and due to the monotonicity of the degrees, the hypergraphs with edges
		$R_1$, $R_2$, and $R_3$ are each left-shifted when restricting to crossing edges (see \cref{lem:local-structure}\,\ref{itm:shifted}). So for
		instance, $i_1j_2 \in R_2$ implies that $i_1i_2 \in R_2$.
		We call an edge $e$ of~$R$ \emph{supported} if $e \cap V(M) \in E(H)$, and denote by $R^+ \subset R$ the $3$-graph of supported edges.  Observe
		that the number of unsupported edges of $R$ is at most quadratic in $n$.
		Hence, for sufficiently large~$n$, this quantity can be bounded by
		\begin{align}\label{equ:supported}
			e(R \sm R^+) < \frac{\gamma}{4}\binom{n}{3} \,.
		\end{align}
		
		As it turns out, a few triples of matching edges may exhibit a rather extrovert degree structure.
		We capture this by calling a triple $T$ \emph{expanding}, if~$V(T)$ contains three pairwise disjoint edges of $H$ whose union intersects~$I$ in at most $2$ vertices, and \emph{steady} otherwise.
		Fortunately, there cannot be too many expanding triples due to the maximality of the matching, which was already
		observed by {\L}uczak and Mieczkowska~\cite{LM14}*{Claim~4}.
		For the sake of completeness, we spell out their argument.
		\begin{claim}\label{cla:bad}
			No three disjoint triples of matching edges are expanding.
		\end{claim}
		\begin{proof}
			Suppose that there exist $9$ disjoint edges
			$\{(i_\ell,j_\ell,k_\ell)\colon 1 \leq \ell \leq 9\}\subset M$
			such among their vertices one can find
			a set of $9$ pairwise disjoint edges $H'\subset H$,
			which do not cover the vertices $i_3$, $i_6$, and $i_9$.
			
			Without loss of generality, we may assume $i_3<i_6<i_9$.
			So in particular, we have $i_6<i_9<j_9<k_9$.
			Since $e=i_9j_9k_9\in M$ is an edge of the left-shifted hypergraph~$R$, it follows that $e'=i_6i_9j_9$ is also an edge
			of~$R$ by considering a $(k_9,i_6)$-shift for $e$. Similarly, the fact that $i_3<i_6<i_9<j_9$
			and considering an $(j_9,i_3)$-shift of $e'$ tells us that $e''=i_3i_6i_9$ is in $R$. Therefore, we find $10$ pairwise
			disjoint edges $H'' =H'\cup\{e''\}\subset H$.
			
			Furthermore, since edges from $R_1\cup R_2$
			have large degrees,  all edges from $H''$ which belong to $R_1\cup R_2$
			can be simultaneously extended to disjoint edges of $R$ by
			adding  vertices of $W$. This leads to
			a matching~$M' \subset R$ of size $|M|+1$
			contradicting the maximality of~$M$.
		\end{proof}
		
		As a consequence of \cref{cla:bad}, there exist six or fewer edges in
		the matching $M$ such that each expanding triple contains one of these
		edges.
		Hence the number of expanding triples does not exceed $6\binom{|M|}{2}$.
		Note that we can crudely acknowledge the contribution of such a triple~$T$ in the bound~\eqref{equ:global-bound-on-R-cup-B} by
		\begin{align*}
			e_1(T)  + e_2(T) + e_3(T) + \Big(\frac{1}{\sigma}-3\Big)^3 \leq \bigg(9\frac{\binom{|W|}{2}}{\binom{|M|-1}{2}}+27\frac{|W|}{|M|-2}+27+\Big(\frac{1}{\sigma}-3\Big)^3\bigg)  \,,
		\end{align*}
		which is bounded by a constant independent of $n$.
		Consequently, their total contribution again quadratic in $n$ and thus limited to
		\begin{align}\label{equ:expandingT}
			\sum_{T\ \text{expanding}}\bigg(e_1(T)  + e_2(T) + e_3(T) + \Big(\frac{1}{\sigma}-3\Big)^3\bigg)
			< \frac{\gamma}{4}\binom{n}{3}\,.
		\end{align}
		
		For the remainder of the argument, we
		fix a steady triple
		\begin{align*}
			T = \{M_1,M_2,M_3\}
		\end{align*}
		of matching
		edges in $M$ that maximises the sum $e_1(T) + e_2(T) + e_3(T)$.
		Our plan is to apply \cref{lem:local-structure} to bound the degrees of $H$ restricted to $T$.
		The following definitions frame the quantities of interest.
		For $v$ and $uv\in\crs(T)$, we set
		\begin{align*}
			r(v)  & = \frac{\deg_{R^+}(v;W)}{\binom{|W|}{2}}\,, &
			r(uv) & = \frac{\deg_{R^+}(uv;W)}{|W|} \,,            \\
			b(v)  & =   \frac{\deg_B(v;W)}{\binom{|W|}{2}}\,,   &
			b(uv) & =  \frac{\deg_B(uv;W)}{|W|}\,,                \\
			q_1   & = \sum_{v\in\crs{(T)}} r(u) + b(v) \,,      &
			q_2   & = \sum_{uv \in \crs{(T)}} r(uv)  + b(uv)\,.
		\end{align*}
		Consequently, in view of the estimates~\eqref{equ:supported} and~\eqref{equ:expandingT}
		we can recontextualise the right-hand side of  inequality~\eqref{equ:global-bound-on-R-cup-B} as
		\begin{align*} 
			e(R\cup B)
			& < \binom{|M|}{3}\cdot \bigg(e_1(T) + e_2(T) + e_3(T) + \Big(\frac{1}{\sigma}-3\Big)^3\bigg) + \frac{3\gamma}{4}\binom{n}{3} \\
			& =
			\binom{|M|}{3} \cdot \bigg( \frac{\binom{|W|}{2}}{\binom{|M|-1}{2}}  q_1 + \frac{\binom{|W|}{2}}{|M|-2} q_2+e_3(T)+\Big(\frac{1}{\sigma}-3\Big)^3 \bigg) + \frac{3\gamma}{4}\binom{n}{3}
		\end{align*}
		
		Moreover, we set and note
		\begin{align*}
			\beta = \frac{1}{\sigma}-3\geq 1\,,
		\end{align*}
		where the lower bound follows from $\sigma\leq 1/4$.
		Since $|M| = \sigma n$ and $|W| = (1-3\sigma) n$, for sufficiently large $n$, we may approximate
		\begin{align*}
			\frac{\tbinom{|W|}{2}}{\tbinom{|M|-1}{2}} \leq \beta^2 + \frac{\gamma}{8\cdot 9}\,,
			\qquad
			\frac{|W|}{{|M|-2}} \leq \beta + \frac{\gamma}{8\cdot 27}\,,
			\qquad\text{and}\qquad
			\binom{|M|}{3}\leq\sigma^3\binom{n}{3}\,.
		\end{align*}
		Therefore, we arrive at
		\begin{align*}
			e(R\cup B)
			<
			\Big(\sigma^3\cdot\left(\beta^2 q_1 + \beta q_2 + e_3(T)+\beta^3\right)+\gamma\Big)\binom{n}{3}\,.
		\end{align*}
		To finish the proof of \cref{lem:matching}, we aim to show that
		\begin{align*}
			\sigma^3\left(\beta^3 + \beta^2 q_1 + \beta q_2 + e_3(T)\right) \leq \frac{5}{8}\,.
		\end{align*}
		Consequently, it is left to verify the structural assumptions~\ref{itm:shadow}--\ref{itm:many-red-triples} of
		\cref{lem:local-structure} for the maximal triple $T$ to conclude the proof of \cref{lem:matching}.
		
		To be precise, we shall apply \cref{lem:local-structure} with $R_1$, $R_2$, and $R_3$ restricted to $T \subset M$ and their union playing the rôle of $R$.
		Moreover, the down-closure $B$ restricted to $T$ gives rise to $B_1$, $B_2$, and $B_3$.
		Now let us verify the assumptions of \cref{lem:local-structure}.
		To begin, note that parts~\ref{itm:shadow} and \ref{itm:distinguishable} follow immediately from the definition. Part~\ref{itm:expanding} captures the steadiness of~$H$.
		Moreover, part~\ref{itm:crossing} for the blue edges follows because $R_3$ and $B_3$ are distinguishable and $T \subset M \subset R$.
		Also, part~\ref{itm:shifted} follows as noted after the definitions of $R_1$, $R_2$, and $R_3$ for $R$ and it is inherited from the right-shiftedness of all blue triples for $B$.
		
		Note that, the functions $r(\cdot)$ and $b(\cdot)$ are restricted to  $R_1$, $R_2$ as well as $B_1$, $B_2$, which act as their support anyways. Consequently, the following claim verifies assumption~\ref{itm:weights} of \cref{lem:local-structure}.
		
		\begin{claim}\label{rem:constraints}
			For $u,v,w \in V(T)$, we have
			\begin{align*}
				r(u)  + b(v)  & \leq 1\,, &
				r(uv) + b(u)  & \leq 1\,,   \\
				r(u)  + b(uw) & \leq 1\,, &
				r(uv) + b(uw) & \leq 1\,.
			\end{align*}
			Moreover, $b(i_\ell) \leq b(j_\ell) \leq b(k_\ell)$ for every $1 \leq \ell \leq 3$.
		\end{claim}
		
		\begin{proof}
			We focus on the first two cases, the others follow similarly.
			Note that since~$R$ and $B$ are distinguishable, there are no two vertices $x,y$ such that $uxy$ is in $R$ and $vxy$ is in $B$.
			So in particular, $ {\deg}_R(u;W) +  {\deg}_B(v;W) \leq \binom{|W|}{2}$, which gives
			$r(u)+ b(v) \leq 1$.
			
			Moreover, the distinguishability
			implies that the blue link of any vertex $u$ can only use vertices in the complement of the red neighbourhood of a pair $uv$ and the second inequality $r(uv) + b(u)  \leq 1$ follows.
			
			Finally, we have $b(i_\ell) \leq b(j_\ell) \leq b(k_\ell)$ for every $1 \leq \ell \leq 3$, since the blue triples are right-shifted and $i_\ell < j_\ell < k_\ell$.
		\end{proof}
		
		We conclude with the three structural assumptions of \cref{lem:local-structure}.
		Without loss of generality, suppose that $M_\ell = (i_\ell,j_\ell,k_\ell)$ for $1\leq \ell \leq 3$ and the following claim  yields part~\ref{itm:blue-star}.
		
		\begin{claim}\label{cla:red-shadow}
			For any two matching edges $(i_\ell,j_\ell,k_\ell)$ and $(i_{p},j_{p},k_{p})$ of $T$ with ${k_\ell < k_{p}}$, the edges of $\partial B$ between the two form a star centred in $k_p$.
		\end{claim}
		\begin{proof}
			Since $k_\l<k_p$, the edge $j_pk_\ell$ is in $\partial R$.
			(Otherwise, we could perform a $(k_p,k_\ell)$-shift on~$\partial R$ replacing $j_pk_p$ with $j_pk_\ell$.)
			Since~$\partial R$ is left-shifted, it follows that $j_\ell j_p$, $i_\ell j_p$, $i_p j_\ell$, and~$i_\ell i_p$ are also in $\partial R$.
			Note that this implies in particular that the only possible edges of $\partial B$ are incident to $k_p$.
		\end{proof}
		
		The next observation, proved by {\L}uczak and Mieczkowska~\cite{LM14}*{Claim~5}, is a consequence of steadiness and implies part~\ref{itm:red-pairs}.
		
		\begin{claim}\label{cla:steady}
			The singletons of $R_1$ in $T$ are in $I$.
			Moreover, $R_2$ has at most $5$ pairs between any two matching edges of $T$ with equality if and only if all $5$ pairs intersect~$I$.
		\end{claim}
		\begin{proof}
			For the first part, let $j_1<j_2<j_3$ and assume
			that some $j_\ell$ is in $R_1$.
			Then, since the singletons of $H$ are shifted to the left, we have $i_1$, $j_1 \in R_1$.
			Consequently, $T$ is an expanding triple because of the singletons $i_1$, $j_1\in R_1$, and the matching edge $i_2j_2k_2\in M\subseteq H$,
			a~contradiction.
			
			Let us assume by contradiction that $6$ pairs of $H$ are contained in
			$\{i_1,j_1,k_1,i_2,j_2,k_2\}$.
			Recall that $R_2$ is left-shifted in the crossing sense.
			Thus $j_1j_2 \in R_2$ and at least one
			of the edges~$i_1k_2$ or $i_2k_1$ is in $R_2$, say it is $i_1k_2$.
			Then, $T$ is {expanding} because of the edges $j_1j_2$, $i_1k_2$, and
			$i_3j_3k_3$.
			Consequently, every pair must intersect $I$, and there are only $5$ such pairs.
		\end{proof}
		
		Our last claim, for the time being, implies part~\ref{itm:many-red-triples}.
		
		\begin{claim}\label{claim:last}
			If $|R_3 \cap \crs{(T)}| \geq 20$, then there are at most $4$ red pairs in $R_2$ between every two matching edges of $T$.
			Moreover, if $|R_3 \cap \crs{(T)}|\geq 22$ and $R_2 \cap \crs{(T)}$ contains
			four pairs between any two matching edges in $T$, then none of those pairs
			intersects with $K$.
		\end{claim}
		\begin{proof}
			Suppose that $|R_3 \cap \crs{(T)}| \geq 20$.
			Since at most $19$ triples contain a vertex from $I$ and the red triples are left-shifted in the crossing sense, this implies that the red triple~$j_1j_2j_3$ must be present in $T$.
			If there are at least $5$ red pairs between, say, $M_1$ and $M_2$, then the pairs $i_1k_2$ and~$i_2k_1$ must
			be $R_2$ by \cref{cla:steady}. Together with the triple $j_1j_2j_3$ contradicts the assumption that $T$ is steady.
			
			{Now suppose $|R_3 \cap \crs{(T)}|\geq 22$ and $R_2 \cap \crs{(T)}$ contains four pairs between any two matching edges in $T$.
				First note that due to crossing left-shiftedness, the red pairs $i_1j_2$ and~$i_2j_1$ are present between $(M_1,M_2)$.
				By symmetry the same holds for the pairs $(M_2,M_3)$ and $(M_1,M_3)$.
				It follows that $T$ cannot contain a red triple in $J \cup K$ with two or more vertices in~$K$.
				Indeed, if $k_1k_2j_3$, say, is a red triple, then the edges $i_1j_2$, $i_2j_1$ and $k_1k_2j_3$ contradict the steadiness of $T$.
				Note that this rules out $4$ of potentially $27$ red triples.
				So by assumption all but one of the other red triples must be present.
				Now, for sake of contradiction, assume that there is a red pair $i_2k_1$, say, between $M_1$ and $M_2$.
				By the above, we find that either $j_1k_2j_3$ 
				or~$j_1j_2k_3$ is a red triple.
				In the former case, we obtain expansion by considering $i_3j_2$, $j_1k_2j_3$ and $i_2k_1$.
				In the latter case, expansion follows from $i_1j_3$, $j_1j_2k_3$ and $i_2k_1$.
				But steady as $T$ goes, this gives a raconteurdiction.}
		\end{proof}
		
		All assumptions of \cref{lem:local-structure} beeing verified, this concludes the proof of \cref{lem:matching}.
	\end{proof}

	\subsection{Local structure analysis}
	
	We shall evaluate particular instances of the function appearing in the conclusion of \cref{lem:local-structure}.
	The following bounds can be verified by elementary calculus.
	
	\begin{fact}\label{fact:optimisation}
		For $\sigma\in[1-2^{-1/3},1/4]$, let $\beta = 1/\sigma -3$ and $f_{s,p,t}(\sigma) = \sigma^3\left(\beta^3 + s \beta^2   + p\beta + t\right)$.
		Then $f_{s,p,t}(\sigma) \leq 5/8$  for any triple $(s,p,t)$ among
		$(6,10,23)$, 
		$(6,11,22)$, 
		$(6,12,21)$,
		$(7,10,21)$,
		$(8,10,19)$, 
		$(9,1,27)$,   
		$(9,6,21)$, 
		$(9,6,23)$,
		$(9,7,21)$,
		$(9,8,19)$, and
		$(9,9,17)$.
	\end{fact}
	
	In the proof of \cref{lem:local-structure}, we furthermore employ the following 
	index-related monotonicity 
	\begin{align}\label{eq:optmono}
		f_{s,p+x,t}(\sigma) &\leq f_{s+x,p,t}(\sigma) 
	\end{align}
	for all $x\geq 0$ whenever $\beta\geq 1$ as in the statement of \cref{fact:optimisation}.
	We remark that due to related monotonicities some of the triples in \cref{fact:optimisation} are redundant.
	However, in the interest of a streamlined presentation, we decided to list all instances appearing in the proof.
	
	\begin{proof}[Proof of \cref{lem:local-structure}]
		Let $T = \{M_1 , M_2 , M_3\}$ be the set of triples $M_\ell = (i_\ell,j_\ell,k_\ell)$ as detailed in the statement, and let $R$ and $B$ be the corresponding hypergraphs with singletons $R_1$ and~$B_1$, pairs $R_2$ and $B_2$ as well as triples $R_3$ and $B_3$.
		We interpret $R$ as the `red' graph and~$B$ as the `blue' graph.
		Let $I = \{i_1,i_2,i_3\}$, $J = \{j_1,j_2,j_3\}$, and $K = \{k_1,k_2,k_3\}$.
		
		We analyse the sums of weights $q_1$ and $q_2$ defined in equation~\eqref{eq:defqs} by focusing on the structure between two
		matching edges at a time. For a pair $(M_i,M_{j})$ of distinct matching edges
		in $T$, let
		\begin{align*}
			q_1(M_i,M_{j}) = \sum_{v \in M_i} r(v)  + b(v) \quad \text{and} \quad q_2(M_i,M_j) = \sum_{uv \in \crs{(M_i,M_j)}} r(uv)  + b(uv)\,,
		\end{align*}
		where $\crs{(M_i,M_j)}$ contains the crossing pairs between $M_i$ and $M_j$.
		Note that in $q_1(M_i,M_{j})$, we only count the singleton weights of the first matching edge.
		Setting
		\[
		q(M_i,M_j) = q_1(M_i,M_j) + q_2(M_i,M_j)\,,
		\]
		gives the following identity
		\begin{align*}
			q & = q(M_1,M_2) + q(M_2,M_3) + q(M_3,M_1) + |R_3 \cup B_3|\,.
		\end{align*}
		
		\begin{claim}\label{cla:strahlung}
			For every $(i,j) \in \{(1,2), (2,3), (3,1)\}$,
			we have $q(M_i,M_j) \leq 6$.
			Moreover,
			\begin{enumerate}[label=\nlabel]
				\item \label{itm:strahlung-5+blue}  if $\crs(M_i,M_j)$ has no red pair incident to $k_i$ or $k_j$, then $q(M_i,M_j) \leq 5 + b(k_ik_j)$, and
				\item \label{itm:strahlung-7} if $\crs(M_i,M_j)$ has at most one blue pair, then $q(M_i,M_j)\leq 7-q_1(M_i,M_j)$.
			\end{enumerate}
		\end{claim}
		
		\begin{proof}
			By symmetry, we may focus on the pair $(M_1,M_2)$.
			By assumption~\ref{itm:blue-star}, the crossing pairs of $B_2$ between $M_1$ and $M_2$ form a star centred in $k_\ell$ for $\ell \in \{1,2\}$.
			Let $\ell'=3-\ell$.
			So the blue singleton and pairs that contribute to the sum $q(M_i,M_j)$ are drawn from the following list $i_1$, $j_1$, $k_1$, $i_{\ell'}k_\ell$, $j_{\ell'}k_\ell$, and $k_1k_2$.
			
			First assume that $R$ contains the red pair $j_1j_2$.
			By shiftedness (assumption~\ref{itm:shifted}), $R$ contains the pairs $i_1i_2$, $i_1j_2$, $i_2j_1$.
			By assumption~\ref{itm:red-pairs}, there are no further pairs in $R$ beyond these.
			Moreover, the only red singleton that contributes to $q(M_1,M_2)$ is $i_1$.
			So the red singleton and pairs we are considering here are drawn from $i_1$, $i_1i_2$, $i_1j_2$, $i_2j_1$, $j_1j_2$.
			We may use assumption~\ref{itm:weights} to bound each of the six terms below by at most $1$, and, therefore, we have
			\begin{align*}
				q(M_1,M_2)
				& \leq
				\big(r(i_1) + b(k_1)\big)
				+ \big(r(i_2j_1) + b(j_1)\big) + \big(r(i_1j_2) + b(i_1)\big)                                             \\
				& \qquad+\big(r(j_1j_2) + b(j_{\ell'}k_\ell)\big) + \big(r(i_1i_2) + b(i_{\ell'}k_\ell)\big) + b(k_1k_2)
				\\&  \leq 5 + b(k_1k_2)  \leq 6\,.
			\end{align*}
			
			Now assume that $R$ does not contain the red pair $j_1j_2$.
			By shiftition (assumption~\ref{itm:shifted}), $R$ contains none of the pairs $j_1k_2,j_2k_1,k_1k_2$.
			So the red singletons and pairs considered this time are drawn from $i_1$, $i_1i_2$, $i_1j_2$, $i_1k_2$, $i_2j_1$,  $i_2k_1$.
			As before, we use assumption~\ref{itm:weights} to bound the weights.
			Our pairing depends on the centre of the blue (shadow) star (assumption~\ref{itm:blue-star}).
			If $\ell = 1$, then the blue singleton and pairs contributing to $q(M_i,M_j)$ are drawn from the following list $i_1$, $j_1$, $k_1$, $i_{2}k_1$, $j_{2}k_1$, and $k_1k_2$, and we bound
			\begin{align*}
				q(M_1,M_2)
				& \leq
				\big(r(i_1i_2) + b(i_1)\big) +
				\big(r(i_1) + b(j_1)\big) +
				\big(r(i_2k_1) + b(k_1)\big)      \\
				& \qquad+
				\big(r(i_{2}j_{1}) + b(i_{2}k_1)\big) +
				\big(r(i_{1}j_{2}) + b(j_{2}k_1)\big) +
				\big(r(i_{2}k_1) + b(k_1k_2)\big) \\ &\leq 5 + \big(r(i_{2}k_1) + b(k_1k_2)\big)
				\leq 6\,.
			\end{align*}
			On the other hand, if $\ell = 2$,  then the blue elements to be considered are $i_1$, $j_1$, $k_1$, $i_{1}k_2$, $j_{1}k_2$, and~$k_1k_2$, which allows us to bound
			\begin{align*}
				q(M_1,M_2)
				& \leq
				\big(r(i_1i_2) + b(i_1)\big) +
				\big(r(i_1) + b(j_1)\big) +
				\big(r(i_2k_1) + b(k_1)\big)      \\
				& \qquad+
				\big(r(i_{2}j_{1}) + b(j_{1}k_2)\big) +
				\big(r(i_{1}j_{2}) + b(i_{1}k_2)\big) +
				\big(r(i_{1}k_2) + b(k_1k_2)\big) \\ &\leq 5 + \big(r(i_{1}k_2) + b(k_1k_2)\big)
				\leq 6\,.
			\end{align*}
			Observe that the above inequalities (before the final bound of $6$) also return part~\ref{itm:strahlung-5+blue} of the refined statement.
			
			It remains to show part~\ref{itm:strahlung-7}.
			So we assume that $\crs(M_i,M_j)$ has at most one pair of $B_2$.
			Consequently, the contributing blue elements are $i_1$, $j_1$, $k_1$, and $k_1k_2$ by assumption~\ref{itm:shifted}.
			Again, we first assume that the red pair $j_1j_2$ is present.
			This means that red pairs $i_1k_2$, $i_2k_1$, and~$k_1k_2$ are not present by assumption~\ref{itm:red-pairs}.
			Observe that $r(i_1j_1) + r(i_1j_2) \leq 2\big(1 - b(i_1)\big)$ by assumption~\ref{itm:weights}.
			More generally, we may estimate
			\begin{align*}\label{equ:knuth-10:29}
				q(M_1,M_2)
				& \leq \big(r(i_1j_1) + r(i_1j_2)\big) + \big(r(j_1i_2) + r(j_1j_2)\big) + b(k_1k_2) \\ & \qquad+ b(i_1) + b(j_1) + \big(b(k_1) + r(i_1)\big)  \\
				& \leq 2\big(1 - b(i_1)\big) + 2\big(1 - b(j_1)\big) + 1  + b(i_1) + b(j_1) + 1
				\\
				& =6 - b(i_1) - b(j_1)                                                               \\
				& =6 - \big(b(i_1) + b(j_1) + b(k_1) + r(i_1)\big) + \big(b(k_1) + r(i_1)\big)       \\
				& \leq 7 - q_1(M_1,M_2)\,.
			\end{align*}
			On the other hand, we have the case where the red pair $j_1j_2$ is not present.
			Note that
			\[
			r(i_2k_1) \leq 1- b(k_1) \leq 1 - b(j_1)
			\]
			by assumption~\ref{itm:weights}.
			So we may bound
			\begin{align*}\label{equ:knuth-10:30}
				q(M_1,M_2)
				& \leq
				\big(r(i_1i_2) + r(i_1j_2)\big) + \big(r(i_2j_1) +  r(i_2k_1)\big)  +  \big(r(i_1k_2) + b(k_1k_2)\big) \\ &\qquad + b(i_1) + b(j_1) +  \big(  r(i_1) + b(k_1) \big)  \\
				& \leq 2\big(1 - b(i_1)\big) + 2\big(1 - b(j_1)\big) + 1  + b(i_1) + b(j_1) + 1
				\\
				& =6 - b(i_1) - b(j_1)                                                                                \\
				& \leq 7 - q_1(M_1,M_2) \,.   \qedhere
			\end{align*}
		\end{proof}
		
		Without loss of generality, we assume for the rest of the proof that $k_1 < k_2 < k_3$.
		
		\begin{claim}\label{cla:at-most-3-blue-triples}
			Each of the blue triples of $B$ contains both $k_2$ and $k_3$.
			In particular, $B$ has at most $3$ blue triples.
		\end{claim}
		\begin{proof}
			By assumption~\ref{itm:blue-star} all edges of $B_2$ are incident with $k_2$ or~$k_3$.
			So a blue triple that does contain only one of $k_2$ and $k_3$ would yield a blue shadow edge that contradicts this.
		\end{proof}
		
		Let us now turn to a more involved discussion of the interaction between the number of blue pairs and the triples.
		
		\begin{claim}\label{cla:counting-red-and-blue-triples}
			We have the following.
			\begin{enumerate}[label=\nlabel]
				\item\label{it:2blue} If $|B_2| = 2$, then $|R_3| \leq 22$. Moreover, if $|R_3| =22$, then $B_2$ contains only two of the pairs $k_1k_2$, $k_1k_3$, and $k_2k_3$.
				\item\label{it:3blue} If $|B_2| = 3$, then $|R_3\cup B_3| \leq 21$.
				\item\label{it:4blue} If $|B_2| \geq 4$, then $|R_3\cup B_3| \leq 19$.
				\item\label{it:8blue} If $|B_2| \geq 6$ and $B_2$ contains $k_1k_2$, $k_1k_3$, and $k_2k_3$, then $|R_3\cup B_3| \leq 17$.
			\end{enumerate}
		\end{claim}
		\begin{proof}
			The proof relies on a somewhat tedious case distinction.
			We employ the fact, that the blue pairs are right-shifted (assumption~\ref{itm:shifted}) and form stars within pairs of triples of $T$ that are centred in $k_1$, $k_2$, and $k_3$ (assumption~\ref{itm:blue-star}).
			This reduces the number of the  possible configurations of blue pairs that need to be considered.
			The claim then follows by counting how many red triples in $R_3$ are forbidden by a blue pair from $B_2$ due to assumption~\ref{itm:distinguishable}.
			
			Obviously, one such pair excludes three red triples.
			For two blue pairs, we arrive, without loss of generality, at one of the following situations.
			\begin{enumerate}[label=\rmlabel]
				\item A blue star of size $2$ between $M_1$ and $M_2$, which results in six missing red triples.
				\item Two blue pairs spanning $k_1$, $k_2$, and $k_3$, which results in five missing red triples.
			\end{enumerate}
			Evidently, if $|R_3| = 22$, then we are in the latter case as desired.
			This verifies assertion~\ref{it:2blue}.
			
			The same reasoning reveals that three blue shadow pairs exclude at least~$7$ red triples in~$R$.
			Indeed, three blue pairs lead to one of the following four constellations.
			\begin{enumerate}[label=\rmlabel] \addtocounter{enumi}{2}
				\item One blue star of size $3$ between $M_1$ and $M_2$, which forbids nine red triples.
				
				\item One blue star of size $2$ centred in $k_1$ between $M_1$ and $M_2$ and one blue pair $k_2k_3$,
				which forbids seven red triples.
				
				\item One blue star of size $2$ centred in $k_1$ between $M_1$ and $M_2$ and one blue pair $k_1k_3$, which forbids seven red triples.
				
				\item The three blue pairs spanned by $k_1$, $k_2$, and $k_3$, which forbids seven red triples.
			\end{enumerate}
			Moreover, if $|B_2|=3$, there is at most one blue triple.
			Consequently, we have at most $27-7+1=21$ crossing triples in $R \cup B$ in total, and this yields assertion~\ref{it:3blue}.
			
			For the proof of part~\ref{it:4blue}, suppose that there are at least $4$ blue shadow pairs.
			Due to assumptions~\ref{itm:shifted} and~\ref{itm:blue-star}, one of the following four subconfigurations materialises for the blue pairs.
			\begin{enumerate}[label=\rmlabel]\addtocounter{enumi}{6}
				\item A blue star of size $3$ between two matching edges, which results in $9$ missing red triples.
				\item Two blue stars of size $2$, each between two matching edges with a common centre, which results in $8$ missing red triples.
				\item Two blue stars of size $2$, each between two matching edges with distinct centres, which results in both subcases leads to $10$ missing red triples.
				\item \label{case:4} One blue star of size $2$ between two matching edges and one blue pair between each of the other pairs of matching edges, which results in $9$ missing red triples.
			\end{enumerate}
			In particular, this implies assertion~\ref{it:4blue} if there are no blue triples.
			Moreover, if there is at least one blue triple, then configuration~\ref{case:4} arises.
			This gives assertion~\ref{it:4blue} if there is one blue triple.
			Lastly, if there are at least two blue triples, then by \cref{cla:at-most-3-blue-triples}, there are are at least five blue pairs, which form a $K_{1,1,2}$ and thus forbid $11$ red triples.
			This concludes the discussion of part~\ref{it:4blue}.
			
			For the proof of part~\ref{it:8blue}, suppose that there are at least $6$ blue pairs in $B_2$ among which there are the pairs $k_1k_2$, $k_1,k_3$, and $k_2k_3$.
			We are therefore guaranteed to exhibit subconfiguration~\ref{case:4} as detailed above.
			Since there are $6$ blue edges, one of the following situations occurs (see also \cref{fig:cases,fig:cases-other}).
			
			\begin{figure}
				\captionsetup[subfigure]{labelformat=empty}
				
				\centering
				\begin{subfigure}[b]{0.45\textwidth}
					\centering
					\begin{tikzpicture}
						\node (j1) at (2, 1) [circle, fill=black, inner sep=1.5pt, label=above:$j_1$] {};
						\node (k1) at (4, 1) [circle, fill=black, inner sep=1.5pt, label=above:$k_1$] {};
						
						\node (j2) at (2, 0) [circle, fill=black, inner sep=1.5pt, label=above:$j_2$] {};
						\node (k2) at (4, 0) [circle, fill=black, inner sep=1.5pt, label=right:$k_2$] {};
						
						\node (j3) at (2, -1) [circle, fill=black, inner sep=1.5pt, label=below:$j_3$] {};
						\node (k3) at (4, -1) [circle, fill=black, inner sep=1.5pt, label=below:$k_3$] {};
						
						\draw[blue, thick] (k1) -- (k2);
						\draw[blue, thick] (k3) -- (k2);
						\draw[blue, thick] (j3) -- (k1);
						\draw[blue, thick] (j1) -- (k2);
						\draw[blue, thick] (j2) -- (k3);
						\draw[blue, thick] (k1) to[bend right=30] (k3);
					\end{tikzpicture}
					\caption{Configuration \ref{itm:three-blue-stars} with $13$ forbidden red triples.}
				\end{subfigure}
				\qquad \quad
				\begin{subfigure}[b]{0.45\textwidth}
					\centering
					\begin{tikzpicture}
						\node (j1) at (2, 1) [circle, fill=black, inner sep=1.5pt, label=above:$j_1$] {};
						\node (k1) at (4, 1) [circle, fill=black, inner sep=1.5pt, label=above:$k_1$] {};
						
						\node (j2) at (2, 0) [circle, fill=black, inner sep=1.5pt, label=above:$j_2$] {};
						\node (k2) at (4, 0) [circle, fill=black, inner sep=1.5pt, label=right:$k_2$] {};
						
						\node (j3) at (2, -1) [circle, fill=black, inner sep=1.5pt, label=below:$j_3$] {};
						\node (k3) at (4, -1) [circle, fill=black, inner sep=1.5pt, label=below:$k_3$] {};
						
						\draw[blue, thick] (k1) -- (k2);
						\draw[blue, thick] (k3) -- (k2);
						\draw[blue, thick] (j1) -- (k2);
						\draw[blue, thick] (j1) -- (k3);
						\draw[blue, thick] (j2) -- (k3);
						\draw[blue, thick] (k1) to[bend right=30] (k3);
					\end{tikzpicture}
					\caption{Configuration \ref{itm:two-blue-stars} with $12$ forbidden red triples.}
				\end{subfigure}
				
				\caption{Configurations of stars.}
				\label{fig:cases}
			\end{figure}

			\begin{figure}
				\captionsetup[subfigure]{labelformat=empty}
				
				\centering
				\begin{subfigure}[b]{0.325\textwidth}
					\centering
					\begin{tikzpicture}
						\node (i1) at (0, 1) [circle, fill=black, inner sep=1.5pt, label=above:$i_1$] {};
						\node (j1) at (1.5, 1) [circle, fill=black, inner sep=1.5pt, label=above:$j_1$] {};
						\node (k1) at (3, 1) [circle, fill=black, inner sep=1.5pt, label=above:$k_1$] {};
						
						\node (i2) at (0, 0) [circle, fill=black, inner sep=1.5pt, label=left:$i_2$] {};
						\node (j2) at (1.5, 0) [circle, fill=black, inner sep=1.5pt, label=above:$j_2$] {};
						\node (k2) at (3, 0) [circle, fill=black, inner sep=1.5pt, label=right:$k_2$] {};
						
						\node (i3) at (0, -1) [circle, fill=black, inner sep=1.5pt, label=below:$i_3$] {};
						\node (j3) at (1.5, -1) [circle, fill=black, inner sep=1.5pt, label=below:$j_3$] {};
						\node (k3) at (3, -1) [circle, fill=black, inner sep=1.5pt, label=below:$k_3$] {};
						
						\draw[blue, thick] (k1) -- (k2);
						\draw[blue, thick] (k3) -- (k2);
						\draw[blue, thick] (j2) -- (k3);
						\draw[blue, thick] (i2) -- (k3);
						\draw[blue, thick] (k1) to[bend right=30] (k3);
						
						\draw[blue, thick] (j1) -- (k2);
					\end{tikzpicture}
					\caption{First}
				\end{subfigure}
				\hfill
				\begin{subfigure}[b]{0.325\textwidth}
					\centering
					\begin{tikzpicture}
						\node (i1) at (0, 1) [circle, fill=black, inner sep=1.5pt, label=above:$i_1$] {};
						\node (j1) at (1.5, 1) [circle, fill=black, inner sep=1.5pt, label=above:$j_1$] {};
						\node (k1) at (3, 1) [circle, fill=black, inner sep=1.5pt, label=above:$k_1$] {};
						
						\node (i2) at (0, 0) [circle, fill=black, inner sep=1.5pt, label=left:$i_2$] {};
						\node (j2) at (1.5, 0) [circle, fill=black, inner sep=1.5pt, label=above:$j_2$] {};
						\node (k2) at (3, 0) [circle, fill=black, inner sep=1.5pt, label=right:$k_2$] {};
						
						\node (i3) at (0, -1) [circle, fill=black, inner sep=1.5pt, label=below:$i_3$] {};
						\node (j3) at (1.5, -1) [circle, fill=black, inner sep=1.5pt, label=below:$j_3$] {};
						\node (k3) at (3, -1) [circle, fill=black, inner sep=1.5pt, label=below:$k_3$] {};
						
						\draw[blue, thick] (k1) -- (k2);
						\draw[blue, thick] (k3) -- (k2);
						\draw[blue, thick] (j2) -- (k3);
						\draw[blue, thick] (i2) -- (k3);
						\draw[blue, thick] (k1) to[bend right=30] (k3);
						
						\draw[blue, thick] (k1) -- (j2);
					\end{tikzpicture}
					\caption{Second}
				\end{subfigure}
				\hfill
				\begin{subfigure}[b]{0.325\textwidth}
					\centering
					\begin{tikzpicture}
						\node (i1) at (0, 1) [circle, fill=black, inner sep=1.5pt, label=above:$i_1$] {};
						\node (j1) at (1.5, 1) [circle, fill=black, inner sep=1.5pt, label=above:$j_1$] {};
						\node (k1) at (3, 1) [circle, fill=black, inner sep=1.5pt, label=above:$k_1$] {};
						
						\node (i2) at (0, 0) [circle, fill=black, inner sep=1.5pt, label=left:$i_2$] {};
						\node (j2) at (1.5, 0) [circle, fill=black, inner sep=1.5pt, label=above:$j_2$] {};
						\node (k2) at (3, 0) [circle, fill=black, inner sep=1.5pt, label=right:$k_2$] {};
						
						\node (i3) at (0, -1) [circle, fill=black, inner sep=1.5pt, label=below:$i_3$] {};
						\node (j3) at (1.5, -1) [circle, fill=black, inner sep=1.5pt, label=below:$j_3$] {};
						\node (k3) at (3, -1) [circle, fill=black, inner sep=1.5pt, label=below:$k_3$] {};
						
						\draw[blue, thick] (k1) -- (k2);
						\draw[blue, thick] (k3) -- (k2);
						\draw[blue, thick] (j2) -- (k3);
						\draw[blue, thick] (i2) -- (k3);
						\draw[blue, thick] (k1) to[bend right=30] (k3);
						
						\draw[blue, thick] (k1) -- (j3);
					\end{tikzpicture}
					\caption{Third}
				\end{subfigure}
				\hfill
				
				\caption{Configurations of case \ref{itm:one-blue-star} with $13$ forbidden red triples in each.}
				\label{fig:cases-other}
			\end{figure}
			
			\begin{enumerate}[label=\rmlabel] \addtocounter{enumi}{10}
				\item \label{itm:three-blue-stars} Three blue stars of size $2$ between distinct pairs of matching edges with pairwise distinct centres, which forbids $13$ red triples.
				
				\item \label{itm:two-blue-stars} Three blue stars of size $2$ between distinct pairs of matching edges such that two centres are shared, which forbids $12$ red triples.
				
				\item \label{itm:one-blue-star} A blue star of size $3$, a star of size $2$ and a star of size $1$ between distinct pairs of matching edges, which leads to three different configurations, each forbidding $13$ red triples.
			\end{enumerate}
			We are therefore done with part~\ref{it:8blue} if there are at most two blue triples present.
			Lastly, if there are three blue triples, then by \cref{cla:at-most-3-blue-triples}, there are are at least five blue pairs, which form a~$K_{1,1,3}$ and thus forbids $15$ red triples.
		\end{proof}
		
		For the remainder of the argument, we assume for sake of contradiction that the outcome bound of \cref{lem:local-structure} fails.
		We start by excluding an abundance of red triples.
		
		\begin{claim}\label{cl:23}
			There are at most $23$ red triples in $R$.
		\end{claim}
		\begin{proof}
			Suppose there are at least $24$ red triples in $R$.  Consequently, the vertices
			$J \cup K$ host at least $5$ triples.
			It follows that there are two disjoint red triples $e$ and $f$ that together cover~$J \cup K$.
			
			Now suppose that $R$ has a red singleton, which is in $I$ by assumption~\ref{itm:red-pairs}.
			Consequently, assumption~\ref{itm:expanding} is violated as witnessed by that singleton together with
			$e$ and~$f$, which is absurd.
			
			Similarly, if $R$ has a red pair, then one such pair is contained in $I$ and together with
			$e$ and $f$ we arrive at the same contradiction.
			Moreover, since $B$ contains at most $9$ singletons and at most $1$ pair (see
			assertion~\ref{it:2blue} of \cref{cla:counting-red-and-blue-triples}),
			we arrive contrary to our assumption at conclusion of \cref{lem:local-structure} by evaluating \cref{fact:optimisation} for $(s,p,t)=(9,1,27)$.
		\end{proof}
		
		Let us now take a closer look at the number of blue pairs.
		
		\begin{claim}\label{cla:at-most-3-pairs}
			We have $|B_2| \leq 3$.
		\end{claim}
		\begin{proof}
			We distinguish between two cases, depending on the presence of blue pairs.
			Let us first assume that each pair $(M_1,M_2)$, $(M_2,M_3)$, and $(M_3,M_1)$ contains a pair of $B_2$.
			In particular, $k_1k_2$, $k_2k_3$, and $k_3k_1$ are in $B_2$ due to assumption~\ref{itm:shifted}.
			
			To warm up, suppose that $|B_2| \geq 6$.
			Then part~\ref{it:8blue} of \cref{cla:counting-red-and-blue-triples} yields $|R_3\cup B_3|\leq 17$.
			Meanwhile, \cref{cla:strahlung} tells us $q_1+q_2\leq 18$. Since $\beta\geq 1$ and $q_1 \leq 9$,
			the term
			\[
			\sigma^3\left(\beta^3 + \beta^2 q_1 + \beta q_2 + |R_3|\right)
			\]
			is maximised for $q_1=q_2=9$ (see inequality~\eqref{eq:optmono}).
			Therefore, we are done by \cref{fact:optimisation} applied for $(s,p,t)=(9,9,17)$.
			
			Now suppose that $4 \leq |B_2| \leq 5$.
			So $|R_3\cup B_3| \leq 19$ by part~\ref{it:4blue} of \cref{cla:counting-red-and-blue-triples}.
			If $q_1 \leq 8$, then the claim follows from	\cref{fact:optimisation} applied for {$(s,p,t)=(8,10,19)$.}
			Consequently, we may assume that $q_1 \geq 8$.
			
			Since $|B_2| \leq 5$ and a pair of triples of $T$ hosts at most two blue pairs (assumption~\ref{itm:blue-star}), there is a pair of triples $(M_i,M_j)$, which contains at most one pair of $B_2$.
			Moreover, $q_1(M_i,M_j) \geq q_1 - 6 \geq 2$, since the two other pairs of triples have singleton weight at most $3$ each.
			By part~\ref{itm:strahlung-7} of \cref{cla:strahlung}, it follows that $q_1 + q_2\leq 6 + 6 + 7 -2 = 17$.
			We conclude the first case, where each $(M_i,M_j)$ contains a blue pair, by applying \cref{fact:optimisation} with {$(s,p,t)=(9,8,19)$.}
			
			For the second case, suppose that the pair $(M_1,M_2)$ contains no  pair of $B_2$ (the other cases follow analogously).
			Note that in this case there are no blue triples in $B$.
			Moreover, part~\ref{it:4blue} of \cref{cla:counting-red-and-blue-triples} yields $|R_3 \cup B_3|\leq 19$.
			If $q_1 \leq 8$, then we are done by \cref{fact:optimisation} applied with $(s,p,t)=(8,10,19)$ as before.
			If, on the other hand, $q_1 \geq 8$, then we can bound $q_1(M_1,M_2) \geq 8-6 = 2$.
			So again \cref{cla:strahlung} reveals that $q_1 + q_2 \leq 6 + 6 + 7-2$.
			We conclude by  \cref{fact:optimisation} applied with $(s,p,t)=(9,8,19)$ as before.
		\end{proof}
		
		Next, we restrict the number of singletons.
		
		\begin{claim}\label{claim:6single}
			We have singleton weight $q_1 \leq 6$.
		\end{claim}
		\begin{proof}
			Suppose otherwise.
			As before, we first assume that each pair $(M_1,M_2)$, $(M_2,M_3)$, and $(M_3,M_1)$ contains a pair of $B_2$.
			In particular, $B_2$ consists of $k_1k_2$, $k_2k_3$, and $k_3k_1$ by assumption~\ref{itm:shifted} and \cref{cla:at-most-3-pairs}.
			By part~\ref{it:3blue} of \cref{cla:counting-red-and-blue-triples}, it follows that $|R_3 \cup B_3| \leq 21$.
			It follows \cref{cla:strahlung} that $q_1 + q_2 \leq 21-q_1 \leq 15$.
			We may therefore conclude the case where each $(M_i,M_j)$ contains a blue pair using \cref{fact:optimisation} with $(s,p,t)=(9,6,21)$, which is pointwise bounded by $(9,6,23)$.
			
			Now suppose that the pair $(M_1,M_2)$ contains no  pair of $B_2$ (the other cases follow analogously).
			We remark again that in this case there are no blue triples in $B$.
			Since $|B_2| \leq 3$ by \cref{cla:at-most-3-pairs}, there is another pair, say $(M_2,M_3)$, which hosts at most one  pair of $B_2$.
			Let
			\[
			x = q_1(M_1,M_2) + q_1(M_2,M_3) = q_1 - q_1(M_3,M_1)\,.
			\]
			Let us first assume that there are at most $21$ red triples in $R$.
			If $q_1 \geq 7$, then $x \geq 7-3 = 4 $.
			So part~\ref{itm:strahlung-7} of \cref{cla:strahlung} allows us to bound $q_1 + q_2 \leq 14-x+6 \leq 16$.
			We may therefore finish by applying \cref{fact:optimisation} applied with $(s,p,t)=(9,7,21)$ as before.
			If, on the other hand,  $6 < q_1 < 7$, then we obtain the weaker bound of $x \leq 3$.
			In this situation \cref{cla:strahlung} only yields $q_1 + q_2 \leq 14-x+6 \leq 17$.
			Fortunately, this is still sufficient to conclude by \cref{fact:optimisation} applied with $(s,p,t)=(7,10,21)$.
			
			Now, let us assume that there are at least $22$ and, invoking \cref{cl:23}, at most~$23$ red triples in $R$.
			By part~\ref{it:2blue} of \cref{cla:counting-red-and-blue-triples}, there are at most $2$ blue pairs in $B$
			and at most one between a single pair of matching edges.
			Again part~\ref{itm:strahlung-7} of  \cref{cla:strahlung}  and assuming $q_1>6$ leads to
			\[
			q_1 + q_2\leq 21-q_1 < 15\,.
			\]
			Another appeal to \cref{fact:optimisation} applied with {$(s,p,t)=(9,6,23)$} finishes the proof.
		\end{proof}
		
		\begin{claim}\label{claim:B-no-more}
			There are at least $22$ red triples in $R$.
		\end{claim}
		\begin{proof}
			If there are at most~$21$ triples in $R_3\cup B_3$,
			then we are done by \cref{claim:6single} and \cref{fact:optimisation} for $(s,p,t)=(6,12,21)$.
			Consequently, part~\ref{it:3blue} of \cref{cla:counting-red-and-blue-triples} yields $|B_2|\leq 2$, which
			prevents the appearance of any blue triple and by the same consideration  as before,
			we arrive at~$|R_3|\geq 22$.
		\end{proof}

		In view of \cref{cl:23} it remains to address the cases $|R_3|=22$ and $|R_3|=23$.
		First we assume $|R_3|=22$. In this case part~\ref{it:2blue} of
		\cref{cla:counting-red-and-blue-triples} combined with assertion~\ref{itm:strahlung-5+blue} of \cref{cla:strahlung} implies
		\[
		q_1+q_2\leq 6+6+5=17,\,
		\]
		and the proof in this case concludes with invoking \cref{fact:optimisation}
		for $(s,p,t)=(6,11,22)$.
		
		In the second case, when  $|R_3|=23$, the same line of reasoning yields
		$q_1+q_2\leq 6 +5 +5 \leq 16$, and a final call to \cref{fact:optimisation} with
		$(s,p,t)=(6,10,23)$, completes the proof of \cref{lem:local-structure}.
	\end{proof}

	\section{Conclusion}\label{sec:conclusion}
	
	We determined the minimum $d$-degree threshold for $k$-uniform Hamilton cycles when $d=k-3$ by studying an extension of the Erdős--Gallai Theorem for $3$-graphs.
	We believe that a similar approach could be used to tackle the thresholds $\th_{d}^{(k)}$ for $k-d \geq 4$, as suggested by Lang and
	Sanhueza-Matamala~\cite{LS22}*{Conjecture~11.6}.
	This echoes a conjecture of Polcyn, Reiher, Rödl, and Schülke~\cite{PRRS21}\,---\,namely  $\th_{d}^{(k)}$ being determined by~$k-d$.
	In the remainder, we discuss two further avenues of research that appear to be worth exploring.

	\subsection*{Connectivity}
	
	An important part of our proof concerns connectivity in dense hypergraphs.
	Originally an auxiliary concept, the structure and interplay of tight components has become an object of study on its own over the recent years~\cites{GHM19,LL23,LP16}. We therefore suggest to further investigate the extremal behaviour of the function $c_k(\lambda)$, which we define as the limes supremum of the edge density a $k$-graph on $n$ vertices may have without containing a tight component on more than $\lambda \binom{n}{k}$ edges.
	By \cref{lem:connection} and the construction in \cref{fig:constructions}, we have $c_3(1/2) = 5/8$.
	We believe for odd $k$ that $c_k(1/2)$ is attained by the $k$-graphs defined as follows:
	the vertex set consists of disjoint sets $X$ and $Y$ with $|X|\geq |Y|$ and its edges are all $k$-sets but those with~$\lfloor k/2 \rfloor$ vertices in $X$ and  $\lceil k/2 \rceil$ vertices in $Y$.
	In light of the proof of \cref{lem:connection}, a plausible approach to this problem would be to study hypergraph versions of \cref{lem:triangles}, which is a natural question in itself.

	\subsection*{Cycles}
	
	What is the maximal number of edges a $k$-graph on $n$ vertices may have that does not contain a tight cycle of length at least $\ell$?
	For $k=2$, this was answered by Erdős and Gallai~\cite{EG59}, and the extremal construction turns out to be a union of cliques of order at most~$\ell-1$.
	In the hypergraph setting, a similar result was shown by Allen, Böttcher, Cooley, and Mycroft~\cite{ABCM17} for $\ell = o(n)$.
	However, when $\l$ becomes large enough new extremal constructions appear (see \cref{fig:constructions}).
	The study of this phenomenon strikes us as interesting.
	
	Let us define the function $\eg_3(\lambda)$ as the limes supremum of the edge density a $3$-graph on~$n$ vertices may have without containing a cycle of length $\lambda n$.
	Using the approach of Allen, Böttcher, Cooley, and Mycroft~\cite{ABCM17},
	we can distill the fact that $\eg_3(3/4) = 5/8$ from \cref{lem:connection,lem:matching} combined with the construction of \cref{fig:constructions}.
	We are convinced that a more careful analysis of our proof should disclose that this type of construction is sharp for all $\lambda\geq 3/5$.
	In particular, in this regime we expect at most two tight components in the extremal constructions.
	For smaller $\lambda$ however, a more complex picture emerges, since more tight components may arise.
	Consider for instance the complement of the canonical extremal construction for Tur\'an's conjecture for the tetrahedron.
	
	\subsection*{Acknowledgements}
	
	We are grateful to the referee of the manuscript for their many constructive remarks and, most importantly, for pointing out a mathematical oversight on an elusive application of K\H{o}nig's theorem intended to simplify the case analysis in the previous version.
	
	We acknowledge the support of the first author  through the European Union's Horizon~2020 research and innovation programme under
	the Marie Sk\l odowska-Curie Action MIDEHA 101018431.
	Moreover, the third author was supported by the grant 23-06815M of the Grant Agency of the Czech Republic.

\end{document}